\documentclass[reqno]{amsart}

\usepackage[T1]{fontenc}
\usepackage{enumerate, amsmath, amsfonts, amssymb, amsthm, mathrsfs}
\usepackage{bbm}
\usepackage[lmargin=3.5cm,rmargin=3.5cm,top=2.9cm,bottom=2.9cm]{geometry}

\usepackage[all]{xy}
\usepackage{tikz}
\usetikzlibrary{arrows, decorations, shapes}
\usepackage{tikz-cd}
\usepackage{colonequals}
\usetikzlibrary{arrows}
\usepackage[colorlinks, citecolor=blue]{hyperref}
\usepackage[noabbrev,capitalise]{cleveref}
\usepackage{multirow}
\usepackage{ulem}
\usepackage{cancel}
\usepackage{caption} 
\usepackage{enumitem}
\setlist[enumerate]{format=\normalfont}

\newtheorem{Corollary}[equation]{Corollary}
\newtheorem{Lemma}[equation]{Lemma}
\newtheorem{Proposition}[equation]{Proposition}
\newtheorem{Theorem}[equation]{Theorem}

\theoremstyle{definition}
\newtheorem{Definition}[equation]{Definition}
\newtheorem{Example}[equation]{Example}
\newtheorem{Remark}[equation]{Remark}
\newtheorem{Question}{Question}

\numberwithin{equation}{section}

\usepackage{setspace}
\newcommand{\marginparstretch}{0.6}
\let\oldmarginpar\marginpar
\renewcommand\marginpar[1]{\-\oldmarginpar[\framebox{\setstretch{\marginparstretch}\begin{minipage}{\marginparwidth}{\raggedleft\tiny #1}\end{minipage}}]{\framebox{\setstretch{\marginparstretch}\begin{minipage}{\marginparwidth}{\raggedright\tiny #1}\end{minipage}}}}

\title{An introduction to monomorphism categories}

\author[Kvamme]{Sondre Kvamme}
\address{Department of Mathematical Sciences\\ 
        NTNU\\ 
        NO-7491 Trondheim\\ 
        Norway}
\email{sondre.kvamme@ntnu.no}

\usepackage{url}
\begin{document}

\subjclass[2020]{}

\begin{abstract}
This manuscript was written for the Proceedings of the ICRA 2022 in Buenos Aires. It can be divided into four parts: The first part is an introduction to the theory of monomorphism categories, including a short survey on some representation theoretic results. The second part is a summary of some of the main results in \cite{GKKP23} on monomorphism categories. It also includes a new result, namely a classification of all monomorphism categories of finite type over cyclic abelian groups. The third part concerns the relationship between submodule categories and $p$-valuated abelian groups. The last part contains a proof of wildness of a certain monomorphism category, rectifying a statement in the literature.
\end{abstract}

\maketitle

\setcounter{tocdepth}{1}
\tableofcontents

	\section{Introduction}

The study of (separated) monomorphism categories can be traced back to the beginning of the 20th century, with work by Hilton and Miller  \cite{Hil07,Mil04}, and later with work of Birkhoff \cite{Bir35}. It was revitalized by Ringel and Schmidmeier \cite{RS06,RS08,RS08b}, who used Auslander--Reiten theory to study its representation theory. Nowadays it is an active area of research, with both its representation theory, as well as its connections to other areas of mathematics being studied. This manuscript is supposed to serve as an introduction to the theory of monomorphism categories.

Section \ref{Section:Monomorphic representations of quivers} introduces monomorphism categories over Artin algebras for arbitrary finite acyclic quivers and recalls results on functorially finiteness and almost split sequences. The relationship between monomorphism category over selfinjective Artin algebras and Gorenstein projective modules is also explained. Note that commutative uniserial Artin algebras, such as truncated polynomial rings $k[x]/(x^n)$ or cyclic abelian groups $\mathbb{Z}/(p^n)$ of order $p^n$ for $p$ a prime, are selfinjective. Monomorphism categories over such rings have historically been of great importance. 

The monomorphism category over a finite dimensional algebra consists of the representations which are projective as representations over the underlying field. In Subsection \ref{Subsection:Monomorphism categories over fields} we explain how this can be used to give a short proof of functorially finiteness. In Subsection \ref{Subsection:PreviousWork} we survey some of the known results on the representation theory of monomorphism categories. In particular, we consider the representation type, the indecomposable representations, and the Auslander--Reiten quiver.

In Section \ref{Section:An epivalence and the Mimo construction} we summarize some of the work in \cite{GKKP23}. The central role is played by an epivalence between the injective stable category of the monomorphism category, and representations over the stable injective modules category. This gives a bijection on indecomposables, and we explain how the Mimo construction is an inverse to this bijection. We also discuss known applications of this result. In particular, we explain how monomorphism categories of uniserial rings of length $2$ are related to representations of quivers over fields, and how monomorphism categories of two different uniserial rings of length $3$ with the same residue field can be related. We also present a new application by classifying all monomorphism categories over $\mathbb{Z}/(p^n)$ of finite representation type. 

There are many results on the representation theory of monomorphism categories over truncated polynomial ring $k[x]/(x^n)$, but less for cyclic abelian groups $\mathbb{Z}/(p^n)$, even though both are local uniserial rings of the same length. One reason is that $\mathbb{Z}/(p^n)$ is not an algebra over a field, and hence some tools are not available, such as covering theory. One approach which has been successfully applied to $\mathbb{Z}/(p^n)$ is the theory of $p$-valuated groups and valuated trees, as studied in \cite{HRW77a,HRW77,RW79,HRW84,RW99}. In particular, from \cite{RW99} one gets all the indecomposable submodule representations over $\mathbb{Z}/(p^n)$ for $n\leq 5$, which are the cases of finite representation type. However, one needs to work a bit to translate the results from $p$-valuated groups into statements about submodule categories. In Section \ref{Section:ValuatedGroups} we try to make this transition a bit easier by explaining how the category of $p$-valuated groups is equivalent to an additive quotient of the submodule category. We then survey some results on valuated trees. We finish by writing down an explicit description of the indecomposables submodule representations over $\mathbb{Z}/(p^n)$ for $n\leq 5$.  Note that most of them (i.e. the simply presented ones) arise from valuated trees. 

In Section \ref{Section:Correction to the literature} we give a correction to a claim in \cite{Sim02}. More precisely, in \cite[Theorem 1.4]{Sim02} it is stated that the monomorphism category of $\mathbb{A}_4$ over $k[x]/(x^4)$ is not wild. We show that it is in fact wild. The statement and proof was kindly provided by the referee.

There is also a large body of work which we do not mention relating monomorphism categories to other areas of mathematics. For example, submodule categories have connections to finite rank Butler groups and representations of posets \cite{Arn00}, Hall polynomials \cite{Sch12},  Littlewood--Richardson tableaux \cite{KS15,KS18,KST19,KS22,Sch11}, metabelian groups \cite{Bir35,Sch05a}, preprojective algebras \cite{RZ14}, singularity categories \cite{Che11,HM21}, tiled orders \cite{Rum81}, and weighted projective lines \cite{KLM13}. They also satisfy Brauer-Thrall type theorems \cite{AHK21}, a version of the Auslander and Ringel-Tachikawa theorem \cite{Moo10}, and are closely connected to stable Auslander algebras \cite{Eir17,HM21,Haf21a}. Monomorphism categories of finite acyclic quivers are used to describe projective representations \cite{EE05}, flat representations  \cite{EOT04}, Gorenstein projective representations \cite{EEG09,EHS13,LZ13,Zha11,JL21,HB24}, Gorenstein flat representations \cite{DELO21}, and cotorsion pairs \cite{DELO21,EHHS13,EPZ20,HJ19a}. They have been studied using tilting theory \cite{Che11,Zha11}, and are related to weighted projective lines \cite{Len11,KLM13}, as well as topological data analysis \cite{BBOS20}. A generalization of monomorphism categories to quivers with monomial relations has been considered in \cite{LZ17,ZX19}.

Except for Theorem \ref{Theorem:RepTypeUniserial} we do not claim to prove any new results. Most of the statements should be well-known to the experts, albeit not all might exist in a written form. Throughout the paper $k$ is a field, $p$ is a prime number, and $\Lambda$ is an Artin $R$-algebra over some commutative artinian ring $R$. The category of left $\Lambda$-modules is denoted by $\operatorname{Mod}\Lambda$, and the category of finitely generated left $\Lambda$-modules by $\operatorname{mod}\Lambda$. 
 
	\section{Monomorphic representations of quivers}\label{Section:Monomorphic representations of quivers}

	\subsection{Definition and basic properties}\label{Subsection:Definition and basic properties}
	
	Fix a finite acyclic quiver $Q$ with vertices $Q_0$ and arrows $Q_1$. A \textit{representation of} $Q$ \textit{over} $\Lambda$ is a collection $(M_\mathtt{i},h_\alpha)_{\mathtt{i}\in Q_0,\alpha\in Q_1}$ where $M_\mathtt{i}$ is a left $\Lambda$-module for each vertex $\mathtt{i}$ in $Q_0$, and $h_\alpha\colon M_\mathtt{i}\to M_\mathtt{j}$ is a morphism of $\Lambda$-modules  for each arrow $\alpha\colon \mathtt{i}\to \mathtt{j}$ in $Q_1$. The \textit{category of representation} $\operatorname{rep}(Q,\operatorname{Mod}\Lambda)$ has as objects the representations of $Q$ over $\Lambda$, and as morphisms $(M_\mathtt{i},h_\alpha)\to (M'_\mathtt{i},h'_\alpha)$ a collection of morphism $(\varphi_\mathtt{i}\colon M_\mathtt{i}\to M'_\mathtt{i})_{\mathtt{i}\in Q_0}$ such that the following diagram commutes for every arrow $\alpha\colon \mathtt{i}\to \mathtt{j}$
	\[
	\begin{tikzcd}
	M_\mathtt{i}\arrow{r}{h_\alpha}\arrow{d}{\varphi_\mathtt{i}} &M_\mathtt{j}\arrow{d}{\varphi_\mathtt{j}}\\
	M'_\mathtt{i}\arrow{r}{h'_\alpha} &M'_\mathtt{j}.
	\end{tikzcd}
	\]  
 Let $\operatorname{rep}(Q,\operatorname{mod}\Lambda)$ be the full subcategory of $\operatorname{rep}(Q,\operatorname{Mod}\Lambda)$ of finitely generated representations, i.e. all representations $(M_\mathtt{i},h_\alpha)$ where $M_\mathtt{i}$ is finitely generated over $\Lambda$ for each $\mathtt{i}\in Q_0$.
	\begin{Definition}
		The \textit{(separated) monomorphism category} $\operatorname{Mono}(Q,\Lambda)$ is the full subcategory of $\operatorname{rep}(Q,\operatorname{Mod}\Lambda)$ consisting of all representations $(M_\mathtt{i},h_\alpha)$ where
		\[
		h_{\mathtt{i},\operatorname{in}}\colon \bigoplus_{\substack{\alpha\in Q_1\\t(\alpha)=\mathtt{i}}}M_{s(\alpha)}\xrightarrow{(h_\alpha)_\alpha} M_\mathtt{i}.\] 
		is monic for every vertex $\mathtt{i}$. The full subcategory of finitely generated monomorphic representations is denoted by $\operatorname{mono}(Q,\Lambda)$
	\end{Definition}

In this manuscript we are interested in the representation theory of $\operatorname{mono}(Q,\Lambda)$.
 
	\begin{Example}
		If $Q=\mathtt{1}\to \mathtt{2}$ is the $\mathbb{A}_2$-quiver, then we write $\operatorname{Sub}(\Lambda)$ for $\operatorname{Mono}(Q,\Lambda)$ and $\operatorname{sub}(\Lambda)$ for $\operatorname{mono}(Q,\Lambda)$. The category $\operatorname{Sub}(\Lambda)$ consists of pairs $(M,N)$ where $M$ is a left $\Lambda$-module, and $N$ is a submodule of $M$. The study of submodule embeddings dates back to the beginning of the 20th century, see \cite{Hil07,Mil04,Mil05}. 
	\end{Example}
	
	\begin{Example}
		Assume $Q=\mathtt{1}\to \mathtt{2}\to \cdots \to \mathtt{n}$ is the linearly oriented $\mathbb{A}_n$-quiver. Then $\operatorname{Mono}(Q,\Lambda)$ consists of flags $M_1\subseteq M_2\subseteq \cdots \subseteq M_n$ in $\operatorname{Mod}\Lambda$. 
	\end{Example}
	
	\begin{Example}\label{Example:NonLinearlyOrientedA_3}
		Assume $Q=\mathtt{1}\to \mathtt{2}\leftarrow\mathtt{3}$ is a non-linearly oriented $\mathbb{A}_3$-quiver. Then $\operatorname{Mono}(Q,\Lambda)$ consists of triples $(M_1,M_2,M_3)$ where $M_2\in \operatorname{Mod}\Lambda$, and $M_1$ and $M_3$ are submodules of $M_2$ with trivial intersection, e.g. $M_1\cap M_3=(0)$. This is why some authors use the word separated.   
	\end{Example}
	
	\begin{Example}
		Let $Q$ be a finite acyclic quiver such that there is at most one path between each pair of vertices in $Q$. Consider the poset whose elements are the vertices of $Q$, and where $\mathtt{i}\leq \mathtt{j}$ if there is a path from $\mathtt{i}$ to $\mathtt{j}$ in $Q$. Poset representations over $\Lambda$ in the sense of \cite{NR72} (see also \cite{Sch08}) are then precisely the representations $(M_\mathtt{i},h_\alpha)$ of $Q$ over $\Lambda$ for which $h_\alpha$ is a monomorphism for each $\alpha\in Q_1$. If $Q$ is a linearly oriented $\mathbb{A}_n$-quiver, then these are precisely the monomorphic representations defined above. However, not every poset representation is in the monomorphism category for a general quiver. Indeed, this can be seen from Example \ref{Example:NonLinearlyOrientedA_3}, since poset representations do not need to have trivial intersection.
	\end{Example}
	
	We are interested in the representation theory of $\operatorname{mono}(Q,\Lambda)$, the category of finitely generated monomorphic representations. In general, it is not an abelian category, so many classical tools, such as Auslander--Reiten theory, are not directly applicable. On the other hand, one can realize $\operatorname{mono}(Q,\Lambda)$ as a subcategory of a module category of an Artin algebra. Indeed, consider the path algebra $\Lambda Q=\bigoplus_{p\in Q_{\geq 0}}\Lambda$, where $Q_{\geq 0}$ denotes the set of path in $Q$. For a path $p$ and $\lambda\in \Lambda$ let $\lambda p$ denote the corresponding element in $\Lambda Q$. Then the multiplication in $\Lambda Q$ is given by 
	$$(\lambda_1 p_1)\cdot(\lambda_2 p_2)=
	\begin{cases}
	(\lambda_1\lambda_2)p_1p_2, & \text{if}\ s(p_1)=t(p_2) \\
	0, & \text{otherwise}.
	\end{cases}$$
	where $\lambda_1\lambda_2$ is obtained from the multiplication in $\Lambda$ and $p_1p_2$ is the concatenation of the paths $p_1$ and $p_2$. Note that $\Lambda Q$ is again an Artin $R$-algebra. The category $\operatorname{mod}\Lambda Q$ of finitely generated left $\Lambda Q$-modules is equivalent to the category $\operatorname{rep}(Q,\operatorname{mod}\Lambda)$ of representations of $Q$ over $\Lambda$. This can be seen in a similar way as for classical representations for path algebras over fields. We will identify representations over $Q$ and modules over $\Lambda Q$ using this equivalence,. Note that the monomorphism category  $\operatorname{mono}(Q,\Lambda)$ becomes a full subcategory of $\operatorname{mod}\Lambda Q$. 
	
	Next we introduce the homological properties that the monomorphism category satisfies. Let $\Gamma$ be an Artin algebra, and let  $\mathcal{C}$ be a full subcategory of $\operatorname{mod}\Gamma$. 
	\begin{enumerate}
		\item $\mathcal{C}$ is \textit{closed under extensions} if for any exact sequence 
		\[
		0\to M_1\to M_2\to M_3\to 0
		\]
		in $\operatorname{mod}\Gamma$ with $M_1,M_3\in \mathcal{C}$, we must have that $M_2\in \mathcal{C}$.
		\item $\mathcal{C}$ is \textit{closed under submodules} if for any object $M\in \mathcal{C}$ and and any monomorphism $N\to M$ we have that $N\in \mathcal{C}$.
		\item $\mathcal{C}$ is a \textit{torsion-free class} if it is closed under extensions and submodules. 
		\item $\mathcal{C}$ is \textit{contravariantly finite} if for any $M\in \operatorname{mod}\Gamma$ there exists a morphism $g\colon N\to M$ with $N\in\mathcal{C}$ such that any morphism $g'\colon N'\to M$ with $N'\in\mathcal{C}$ must factor through $g$. We call such a morphism $g$ for a \textit{right} $\mathcal{C}$-\textit{approximation}.
		\item $\mathcal{C}$ is \textit{covariantly finite} if for any $M\in \operatorname{mod}\Gamma$ there exists a morphism $g\colon M\to N$ with $N\in\mathcal{C}$ such that any morphism $g'\colon M\to N'$ with $N'\in\mathcal{C}$ must factor through $g$. We call such a morphism $g$ for a \textit{left} $\mathcal{C}$-\textit{approximation}.
		\item $\mathcal{C}$ is \textit{functorially finite} if it is contravariantly finite and covariantly finite.
	\end{enumerate}

	\begin{Theorem}\label{Lemma:MonoExtClosedFunctFinite}
		The monomorphism category $\operatorname{mono}(Q,\Lambda)$ is a functorially finite torsion-free class of $\operatorname{mod}\Lambda Q$.
	\end{Theorem}
	
	\begin{proof}
		Given a short exact sequence $0\to (M_\mathtt{i},h_\alpha)\to (M'_\mathtt{i},h'_\alpha)\to (M''_\mathtt{i},h''_\alpha)\to 0$ of representations of $Q$ over $\Lambda$, we get short exact sequences 
		\[
		\begin{tikzcd}
		0 \arrow{r} & \bigoplus_{\substack{\alpha\in Q_1\\t(\alpha)=\mathtt{i}}}M_{s(\alpha)}\arrow{r}\arrow{d}{h_{\mathtt{i},\operatorname{in}}} &\bigoplus_{\substack{\alpha\in Q_1\\t(\alpha)=\mathtt{i}}}M'_{s(\alpha)}\arrow{d}{h'_{\mathtt{i},\operatorname{in}}}\arrow{r} & \bigoplus_{\substack{\alpha\in Q_1\\t(\alpha)=\mathtt{i}}}M''_{s(\alpha)}\arrow{d}{h''_{\mathtt{i},\operatorname{in}}}\arrow{r} & 0\\
		0 \arrow{r}& M_\mathtt{i}\arrow{r} &M'_\mathtt{i}\arrow{r} & M''_\mathtt{i} \arrow{r}& 0.
		\end{tikzcd}
		\]
		for each vertex $\mathtt{i}$ in $Q$. Using this and the fact that monomorphisms are closed under extensions and under inclusions, we get that $\operatorname{mono}(Q,\Lambda)$ is closed under extensions and under submodules. The fact that $\operatorname{mono}(Q,\Lambda)$ is functorially finite holds by \cite[Theorem 3.1]{LZ13} (note that the theorem is only stated for representations over finite-dimensional algebras, but the proof also works for Artin algebras).
	\end{proof}

 Functorially finiteness of the monomorphism category was first proven in \cite{RS08} for the $\mathbb{A}_2$-quiver. It was extended to linearly oriented $\mathbb{A}_n$-quivers for any $n\geq 1$ in \cite{XZZ14}. It was proven for a general finite acyclic quiver in \cite{LZ13}. In the next subsection we give a short self-contained proof of this fact for monomorphic representations over finite-dimensional algebras 
	
	In \cite{AS81} Auslander and Smalø considered almost split sequences in extension closed subcategories of module categories. In particular, they introduced the notion of a subcategory having almost split sequences. Their motivation was to extend Auslander--Reiten theory from module categories to more general settings. In particular, they show
	\begin{Theorem}\label{Theorem:AuslanderSmalø}
		Let $\mathcal{C}$ be a functorially finite extension closed subcategory of $\operatorname{mod}\Gamma$ for an Artin algebra $\Gamma$. Then $\mathcal{C}$ has almost split sequences.
	\end{Theorem}
	
	\begin{proof}
		This follows from \cite[Theorem 2.4]{AS81}. 
	\end{proof} 
	
	\begin{Corollary}
		The category $\operatorname{mono}(Q,\Lambda)$ has almost split sequences.
	\end{Corollary}
	
	\begin{proof}
		This follows from Theorem \ref{Lemma:MonoExtClosedFunctFinite} and Theorem \ref{Theorem:AuslanderSmalø}.
	\end{proof}
	
	We end with the observation that the monomorphism category can be identified with a different well-studied subcategory when $\Lambda$ is a selfinjective algebra. 
	
	\begin{Theorem}\label{Theorem:GorensteinProjMono}
		Assume $\Lambda$ is a selfinjective algebra. The following hold:
		\begin{enumerate}
			\item\label{Theorem:GorensteinProjMono:1} $\Lambda Q$ is $1$-Iwanaga-Gorenstein, i.e. has injective dimension $1$ as a left and as a right module over itself.
			\item\label{Theorem:GorensteinProjMono:2} The monomorphism category $\operatorname{mono}(Q,\Lambda)$ is equal to the subcategory 
			\[
			\operatorname{Gproj}\Lambda Q=\{M\in \operatorname{mod}\Lambda Q\mid \operatorname{Ext}^1_{\Lambda Q}(M,\Lambda Q)=0  \}
			\]
			of Gorenstein projective $\Lambda Q$-modules.
		\end{enumerate}
	\end{Theorem}
	
	\begin{proof}
		The first statement follows from \cite[Theorem 3.5]{EEGI07}, and the second statement from \cite[Corollary 6.1]{LZ13}.
	\end{proof}
	
	The representation theory of Gorenstein projective modules has been studied in \cite{Rin13,CSZ18} for particular classes of algebras. The result above tells us that the representation theory of the monomorphism category is another such an instance, corresponding to classes of algebras of the form $\Lambda Q$ with $\Lambda$ selfinjective.

	\subsection{Monomorphism categories over fields}\label{Subsection:Monomorphism categories over fields} 
	
	Fix a field $k$. Assume $\Lambda$ is a finite-dimensional algebra over $k$. In this case the monomorphic representations over $\Lambda$ can be identified with certain bimodules which are projective as one-sided modules, which we show in this subsection. This makes certain homological properties easier to show. We will use this to give a short proof of functorially finiteness for the monomorphism category. 
	
	To start with we consider representations over the base field $k$. Note that the following result is a consequence of Theorem \ref{Theorem:GorensteinProjMono}, since $kQ$ is hereditary and the Gorenstein projective and the projective modules therefore coincide. For completeness we still include a proof. 
	
	\begin{Proposition}\label{Proposition:MonomorphismOverField}
		The monomorphism category $\operatorname{mono}(Q,k)$ is equal to the subcategory of finite-dimensional projective $kQ$-modules.
	\end{Proposition}
	
	\begin{proof}
		Let $S_\mathtt{i}$ denote the simple right $kQ$-module concentrated at vertex $\mathtt{i}$ and $P_\mathtt{i}$ denote the indecomposable projective right $kQ$-module with top $S_\mathtt{i}$. Then we have projective resolutions
		\[
		0\to \bigoplus_{\alpha\in Q_1, t(\alpha)=\mathtt{i}}P_{s(\alpha)}\xrightarrow{} P_\mathtt{i}\to S_\mathtt{i}\to 0
		\]
		where $P_{s(\alpha)}\to P_\mathtt{i}$ is the canonical morphism induced from the arrow $\alpha\colon s(\alpha)\to \mathtt{i}$. Now let $(M_\mathtt{i},h_\alpha)_{\mathtt{i}\in Q_0,\alpha\in Q_1}$ be a representation of $Q$ over $k$, and let $M$ denote the corresponding left $kQ$-module. Applying $-\otimes_{kQ}M$ to the exact sequence above gives an exact sequence
		\[
		0 \to \operatorname{Tor}_1^{kQ}(S_\mathtt{i},M)\to \bigoplus_{\alpha\in Q_1, t(\alpha)=\mathtt{i}}P_{s(\alpha)}\otimes_{kQ}M \to P_\mathtt{i}\otimes_{kQ}M\to S_\mathtt{i}\otimes_{kQ}M\to 0
		\]
		where $\operatorname{Tor}_1^{kQ}(P_\mathtt{i},M)=0$ since $P_\mathtt{i}$ is projective. Now for any vertex $\mathtt{j}$ we have a canonical isomorphism $P_\mathtt{j}\otimes_{kQ}M\cong M_\mathtt{j}$. Under these isomorphisms the exact sequence above is isomorphic to the exact sequence
		\[
		0 \to \operatorname{Ker}h_{\mathtt{i},\operatorname{in}}\to \bigoplus_{\alpha\in Q_1,t(\alpha)=\mathtt{i}}M_{s(\alpha)}\xrightarrow{h_{\mathtt{i},\operatorname{in}}} M_\mathtt{i}\to \operatorname{Coker}h_{\mathtt{i},\operatorname{in}}\to 0
		\]
		Hence $M$ is in the monomorphism category if and only if $\operatorname{Tor}_1^{kQ}(S_\mathtt{i},M)=0$ for all simple right modules $S_\mathtt{i}$. This latter condition is equivalent to $\operatorname{Tor}_1^{kQ}(N,M)=0$ for all  right modules $N$, which is again equivalent to $M$ being flat. But since $M$ is finite-dimensional, being flat and being projective is the same. This proves the result.
	\end{proof}
	
	Now let $\Lambda$ be a finite-dimensional algebra over $k$. We have a canonical morphism of algebras
	\begin{align*}
		\phi\colon \Lambda \to \Lambda Q, \quad \quad \lambda \mapsto \sum_{\mathtt{i}\in Q_0}\lambda e_\mathtt{i}
	\end{align*}
	where $e_\mathtt{i}$ denotes the trivial path at vertex $\mathtt{i}$. The inclusion $k\to \Lambda$ also induces a morphism of algebras $\psi\colon kQ\to \Lambda Q$. Using $\phi$ and $\psi$ we get an isomorphism of algebras
	\[
	\Lambda\otimes_{k}kQ\xrightarrow{\cong} \Lambda Q \quad \quad \lambda\otimes p\mapsto \phi(\lambda)\cdot \psi(p). 
	\]
	where the multiplication in $\Lambda\otimes_{k}kQ$ is given by $(\lambda_1\otimes p)\cdot (\lambda_2\otimes q)=\lambda_1\cdot \lambda_2\otimes p\cdot q$. A module over $\Lambda\otimes_k kQ$ is the same as a $k$-vector space $M$ with a left action of $\Lambda$ and a left action of $kQ$ making it into a left $\Lambda$-module and a left $kQ$-module, respectively, and such that the two actions commute, i.e. 
	\[
	\lambda\cdot p\cdot m=p\cdot \lambda\cdot m
	\]
	for $\lambda\in \Lambda$ and $p\in kQ$ and $m\in M$. Note that this is equivalent to $M$ being a $(kQ,\Lambda^{\operatorname{op}})$-bimodule, where $\Lambda^{\operatorname{op}}$ denotes the opposite ring of $\Lambda$. For a module $M$ over $\Lambda\otimes_{k}kQ\cong \Lambda Q$ we let ${}_{kQ}|M$ denote its restriction to $kQ$. Note that $M\mapsto {}_{kQ}|M$ induces an exact functor
	\[
	\operatorname{mod}\Lambda Q\to \operatorname{mod}k Q.
	\]
	The preimage of the projective modules are precisely the monomorphic representations.
	\begin{Corollary}\label{Corollary:MonomorphismOverAlgOverField}
		Let $\Lambda$ be a finite-dimensional $k$-algebra, and let $M$ be a finite-dimensional left $\Lambda Q$-module. Then $M$ is contained in $\operatorname{mono}(Q,\Lambda)$ if and only if ${}_{kQ}|M$ is projective.
	\end{Corollary}
	
	\begin{proof}
		Note that $M$ is in $\operatorname{mono}(Q,\Lambda)$ if and only if ${}_{kQ}|M$ is in $\operatorname{mono}(Q,k)$. The claim follows now from Proposition \ref{Proposition:MonomorphismOverField}.
	\end{proof}
	
	We can use Corollary \ref{Corollary:MonomorphismOverAlgOverField} and the theory of torsion-free classes to give a short proof of the functorially finiteness of $\operatorname{mono}(Q,\Lambda)$.
	
	\begin{proof}[Proof of Theorem \ref{Lemma:MonoExtClosedFunctFinite} for finite-dimensional algebras.] 
		
		Since the functor $M\mapsto {}_{kQ}|M$ is exact and projective modules are closed under extensions, it follows that $\operatorname{mono}(Q,\Lambda)$ is closed under extensions. Also, since $kQ$ is hereditary, the projective $kQ$-modules are closed under subobjects, so the same holds for $\operatorname{mono}(Q,\Lambda)$. Hence, $\operatorname{mono}(Q,\Lambda)$ is a torsion-free class,  i.e. there exists a torsion pair $(\mathcal{T},\mathcal{F})$ where $\mathcal{F}=\operatorname{mono}(Q,\Lambda)$. By definition of torsion pairs, for any $\Lambda Q$-module $M$ there exists an exact sequence $$0\to t(M)\to M\to f(M)\to 0$$ where $t(M)\in \mathcal{T}$ and $f(M)\in \mathcal{F}$. Since $\operatorname{Hom}_{\Lambda Q}(t(M),F)=0$ for any $F\in \mathcal{F}$, it follows that $M\to f(M)$ is a left $\operatorname{mono}(Q,\Lambda)$-approximation. 
		
		It remains to show that any $\Lambda Q$-module has a right $\operatorname{mono}(Q,\Lambda)$-approximation. By \cite[Proposition 4.7 (b)]{AS80} it suffices to show that there exists $N\in \operatorname{mono}(Q,\Lambda)$ such that  $\operatorname{mono}(Q,\Lambda)=\operatorname{sub}N$, where 
  \[
  \operatorname{sub}N\colonequals\{M\in \operatorname{mod}\Lambda Q\mid \text{there exists a monomorphism } M\to N^m \text{ for some } m>0\}.
  \]
  We claim that this holds for $N=D(\Lambda)\otimes kQ$ where $D(\Lambda)=\operatorname{Hom}_k(\Lambda,k)$. Indeed, it is clear that $D(\Lambda)\otimes kQ\in \operatorname{mono}(Q,\Lambda)$ since 
		\[
		{}_{kQ}|D(\Lambda)\otimes kQ\cong kQ^{n}
		\]
		which is projective, where $n$ is the dimension of $D(\Lambda)$. Now let $M\in \operatorname{mono}(Q,\Lambda)$. Then the restriction ${}_{kQ}|M$ is projective, and in particular there exists an inclusion of $kQ$-modules $i\colon {}_{kQ}|M\to kQ^m$  for some integer $m$. Next note that ${}_{kQ}|(-)\colon \operatorname{mod}\Lambda Q\to \operatorname{mod}kQ$ has right adjoint
		\[
		\operatorname{Hom}_k(\Lambda,-)\colon \operatorname{mod}kQ\to \operatorname{mod}\Lambda Q
		\]
		Hence $i$ corresponds to a morphism of $\Lambda Q$-modules $\tilde{i}\colon M\to \operatorname{Hom}_k(\Lambda,kQ^m)$  given by $\tilde{i}(m)(\lambda)=i(\lambda\cdot m)$. Clearly, $i$ being a monomorphism implies $\tilde{i}$ is a monomorphism. Since we have canonical isomorphisms
		\[
		\operatorname{Hom}_k(\Lambda,kQ^m)\cong D(\Lambda)\otimes_k (kQ^m)\cong (D(\Lambda)\otimes_k kQ)^m
		\]
		of $\Lambda Q$ modules, the claim follows.
	\end{proof}
	
	\begin{Remark}
		The characterization in Corollary \ref{Corollary:MonomorphismOverAlgOverField} does not hold for an arbitrary Artin algebra $\Lambda$. Indeed, if $Q$ is the quiver with one vertex and no arrows and $\Lambda=R$ is a commutative artinian ring, then the monomorphism category must be equal to $\operatorname{mod}R$ itself. This coincides with the category of finitely generated projective $R$-modules if and only if $R$ is semisimple, which for example is not true for $R=\mathbb{Z}/(p^n)$. 
		
		On the other hand, the proof of Theorem \ref{Lemma:MonoExtClosedFunctFinite} above can be adapted to work for any Artin $R$-algebra $\Lambda$. Indeed, it is still true that for any $M\in \operatorname{mono}(Q,\Lambda)$ there exists a monomorphism $M\to N^m$ for some integer $m$, where $$N=\operatorname{Hom}_R(\Lambda,E)\otimes_R RQ$$ and $E$ is an injective cogenerator in $\operatorname{mod}R$. However, to show this we can no longer rely on Corollary \ref{Corollary:MonomorphismOverAlgOverField}. We sketch two different ways of proving this:
  \begin{enumerate}
      \item One can use \cite[Proposition 4.9 (ii)]{Kva20a} to obtain that the monomorphism category is equal to the category of Gorenstein $P$-projective objects. Now any Gorenstein $P$-projective object is a subobject of a $P$-projective object by definition, see \cite[Definition 2.10]{Kva20a}. In our case the $P$-projective objects are precisely summands of objects of the form $M\otimes_R RQ$, where $M$ is a finitely generated $R$-module $M$. Now $\operatorname{Hom}_R(\Lambda,E)$ being a cogenerator of $\operatorname{mod}R$ implies that $\operatorname{Hom}_R(\Lambda,E)\otimes_R RQ$ is a cogenerator of the $P$-projective objects, which proves the claim.
      \item Consider the ring $S=\prod_{\mathtt{i}\in Q_0}\Lambda$ and the relative top functor
      \[
      \operatorname{top}_Q\colon \operatorname{mod}\Lambda Q\to \operatorname{mod}S \cong \prod_{\mathtt{i}\in Q_0}\operatorname{mod}\Lambda
      \]
      given by
      \[
      \operatorname{top}_Q(M)=(S_\mathtt{i}\otimes_{\Lambda Q}M)_{\mathtt{i}\in Q_0}
      \]
      where $S_i$ is the $(\Lambda Q)^{\operatorname{op}}=\Lambda^{\operatorname{op}}Q^{\operatorname{op}}$-module given by $\Lambda$ at vertex $\mathtt{i}$ and $0$ elsewhere. For each $\mathtt{i}\in Q_0$ choose an inclusion of $\Lambda$-modules
      \[
S_\mathtt{i}\otimes_{\Lambda Q}M\xrightarrow{g_\mathtt{i}}\operatorname{Hom}_R(\Lambda,E)^{n_{\mathtt{i}}}.
      \]
    for some integer $n_\mathtt{i}$.
      By \cite[Lemma 4.1]{GKKP23} we can find a morphism $$g\colon M\to \bigoplus_{\mathtt{i}\in Q_0}\operatorname{Hom}_R(\Lambda,E)^{n_\mathtt{i}}\otimes_R RQ\cong N^m$$
      of $\Lambda Q$-modules satisfying $\operatorname{top}_Q(g)=(g_\mathtt{i})_{\mathtt{i}\in Q_0}$, where $m=\sum_{\mathtt{i}\in Q_0}n_\mathtt{i}$. Since $\operatorname{top}_Q(g)$ is a monomorphism and $M$ is in the monomorphism category, $g$ must be a monomorphism by \cite[Lemma 3.24]{GKKP23}. This proves the claim.
  \end{enumerate}
	\end{Remark}

 \subsection{Previous work}\label{Subsection:PreviousWork}

 Here we summarize some of the work on the representation theory of monomorphism categories.

\subsubsection*{Finite representation type}
 The finite and infinite representation type of $$\operatorname{mono}(\mathtt{1}\to \mathtt{2}\to \cdots \to \mathtt{m},\Lambda)$$ for a commutative uniserial ring $\Lambda$  was determined in \cite{Pla76} (see also \cite[Theorem 1.3]{Sim02}). The classification only depends on $m$ and the length of $\Lambda$, and not the isomorphism class of $\Lambda$.  In \cite[Theorem 4.6]{Lu20} this classification was extended to any finite acyclic quiver when $\Lambda$ was equal to $k[x]/(x^n)$ where $k$ is an algebraically closed field. We give an analogous result for $\Lambda=\mathbb{Z}/(p^n)$ in Theorem \ref{Theorem:RepTypeUniserial} below. A similar classification as in \cite{Pla76} was given in \cite[Corollary 1.6]{BBOS20} for the representation type of $\operatorname{mono}(\mathtt{1}\to \mathtt{2}\to \cdots \to \mathtt{m},k(\mathtt{1}\to \mathtt{2}\to \cdots \to \mathtt{n}))$.

    For the submodule category $\operatorname{sub}(\Lambda)=\operatorname{mono}(\mathtt{1}\to \mathtt{2},\Lambda)$ there are classification results when $\Lambda$ is the path algebra of a quiver \cite[Theorem 3.25]{BBOS20}, and when $\Lambda$ is a selfinjective algebra of finite representation type over an algebraically closed field of characteristic $\neq 2$ \cite[Theorem 7.8]{HM20}. In the case of a commutative artinian uniserial ring $\Lambda$, the classification can be refined by considering the subcategory $\operatorname{sub}_i(\Lambda)$ of $\operatorname{sub}(\Lambda)$ of all monomorphisms $M_1\to M_2$ for which $M_1$ is annihilated by $\mathfrak{m}^i$ for some integer $i\geq 0$, where $\mathfrak{m}$ is the maximal ideal of $\Lambda$. We refer to \cite{Sch05,Sch08,PS11a} for work in this direction.

\subsubsection*{Tame-wild classification}
    By \cite[Theorems 1.4 and 1.5]{Sim02} there is a tame-wild dichotomy for $\operatorname{mono}(\mathtt{1}\to \mathtt{2}\to \cdots \to \mathtt{m},k[x]/(x^n))$ when $k$ is algebraically closed.\footnote{There is a small error in  \cite[Theorems 1.4 and 1.5]{Sim02}, see Theorem \ref{Theorem:RefereeWild}.} For $\mathtt{m}=2$ the assumption on the field being algebraically closed can be dropped by \cite[Proposition 3]{RS06} and the classification result in \cite{RS08b}. Similarly, by \cite[Corollary 1.6]{BBOS20} there is a tame-wild dichotomy for $\operatorname{mono}(\mathtt{1}\to \mathtt{2}\to \cdots \to \mathtt{m},k(\mathtt{1}\to \mathtt{2}\to \cdots \to \mathtt{n}))$. In fact, in both \cite{Sim02} and \cite{BBOS20} they determine precisely for which pairs of integers $(\mathtt{m},n)$ the category is tame or wild. 
    
    A challenging problem is to determine tameness and wildness for monomorphic representations over $\mathbb{Z}/p^n\mathbb{Z}$. By \cite{Pla76} the category $\operatorname{sub}(\mathbb{Z}/(p^n))$ is of finite representation type if and only if $n\leq 5$. The question of whether $\operatorname{sub}(\mathbb{Z}/(p^6))$ is tame is open. It is related to the Birkhoff problem, see the paragraph on computations below. For $n\geq 7$, it was shown in \cite{RS06} that $\operatorname{sub}(\mathbb{Z}/(p^n)))$ is controlled $\mathbb{F}_p$-wild, where $\mathbb{F}_p=\mathbb{Z}/(p)$ is the finite field with $p$ elements. Note that a version of wildness for certain subcategories of $\operatorname{sub}(\mathbb{Z}/(p^n))$ had already been shown in \cite[Chapter 12]{Arn00}, see for example \cite[Example 8.2.6]{Arn00}.
    
\subsubsection*{Structural results}
   For $\Lambda$ commutative and uniserial the period of the Auslander--Reiten translation for a non-injective indecomposable module in $\operatorname{sub}(\Lambda)$ is of order $1$, $2$, $3$, or $6$ \cite[Corollary 6.5]{RS08}.  This is proven using a simple formula for the Auslander--Reiten translation in $\operatorname{sub}(\Lambda)$, which holds for any Artin algebra $\Lambda$  \cite[Theorem 5.1]{RS08}. The formula is extended to $\operatorname{mono}(\mathtt{1}\to \mathtt{2}\to \cdots \to \mathtt{m},\Lambda)$ in \cite{XZZ14}, where they apply it to show that for $\Lambda$ commutative and uniserial the period of the AR-translation divides $m+1$ if $m$ is odd, and $2(m+1)$ if $m$ is even. A similar result holds for connected cyclic selfinjective Nakayama algebra \cite[Corollary 3.6]{XZZ14}. 
    
 If $\Lambda$ is selfinjective over an algebraically closed field and $\operatorname{sub}(\Lambda)$ is of finite representation type, then they show in \cite[Theorem 6.2]{HZ23} that the cardinality of a component of the stable AR-quiver must be divisible by $3$, as long as the tree class of the component is not $\mathbb{D}_4$.  A generalization of this to $\operatorname{mono}(\mathtt{1}\to \mathtt{2}\to \cdots \to \mathtt{m},\Lambda)$ is given in \cite[Theorem 5.7]{HB24}, under the additional assumption that the stable Auslander algebra of $\Lambda$ is semisimple. Furthermore, in \cite{HZ23} they show that a special class of indecomposable objects in $\operatorname{sub}(\Lambda)$, called boundary vertices, gives a lot of information on the stable Auslander--Reiten quiver of $\operatorname{sub}(\Lambda)$. See for example Theorem 4.12, Proposition 4.15, Proposition 5.4 and Theorem 5.7 in \cite{HZ23}. Note that all vertices of the stable AR-quiver are boundary if the stable Auslander algebra of $\Lambda$ is semisimple \cite[Definition 3.1 and Proposition 3.7]{HB24}. For such algebras there is a complete description of the stable AR-quiver of $\operatorname{sub}(\Lambda)$ by \cite[Theorem 5.8]{HB24}.

\subsubsection*{Computations}
By \cite{Pla76} (see also Theorem \ref{Theorem:RepTypeUniserial}) the category $$\operatorname{mono}(\mathtt{1}\to \mathtt{2}\to \cdots \to \mathtt{m},k([x]/(x^n))$$ is representation finite if $n=2$ or $(n,\mathtt{m})\in \{(3,2),(4,2),(5,2),(3,3),(3,4)\}$. The Auslander--Reiten quiver has been computed in these cases. We refer to \cite[Section 6]{RS08b} for the case $(n,\mathtt{m})\in \{(3,2),(4,2),(5,2)\}$, to \cite[Theorem 3.2]{Moo09} for the case $(n,\mathtt{m})=(3,3)$ (alternatively see \cite[Appendix 6.1(ii)]{XZZ14}), and to \cite[Appendix 6.1(iii)]{XZZ14} for the case $(n,\mathtt{m})=(3,4)$. The case $n=2$ can be computed using the results for perfect differential $kQ$-modules in \cite[Theorems 1.1 and 1.3]{RZ17}. To see this, note that for a finite acyclic quiver $Q$ the perfect differential $kQ$-modules in \cite{RZ17} are by definition the $k[x]/(x^2)\otimes_k kQ$-bimodules which are projective as $kQ$-modules. Hence by Corollary \ref{Corollary:MonomorphismOverAlgOverField} they coincide with the monomorphic representations. More generally, the results in \cite{RZ17} can be used to compute $\operatorname{mono}(Q,k[x]/(x^2))$ for any finite acyclic quiver $Q$. Finally, in \cite[Sections 5.1 and 5.3]{LS22} they compute the AR-quiver of $\operatorname{mono}(Q,k[x]/(x^2))$ and $\operatorname{mono}(Q,k[x]/(x^3))$, respectively, for $Q$ any orientation of the $\mathbb{A}_3$-quiver.

By \cite[Theorem 1.4]{Sim02} the category $\operatorname{mono}(\mathtt{1}\to \mathtt{2}\to \cdots \to \mathtt{m},k([x]/(x^n))$ is tame if and only if $(n,\mathtt{m})\in \{(6,2),(4,3),(3,5)\}$\footnote{We were kindly informed by the referee that the pair $(4,4)$ is not tame, contradictory to \cite[Theorem 1.4]{Sim02}. For a proof see Theorem \ref{Theorem:RefereeWild}.}. For $(n,\mathtt{m})=(6,2)$ the AR-quiver was computed in \cite{RS08b}, and further properties were investigated in \cite{DMS19,RS24}. For $(n,\mathtt{m})=(4,3)$ the AR-quiver was computed in \cite[Section 3]{Moo09}, and the possible dimension types were studied in \cite{MS15}. To our knowledge, the AR-quiver for $(n,\mathtt{m})=(3,5)$ has not been computed yet.

Less is known for monomorphic representations over the cyclic abelian group $\mathbb{Z}/(p^n)$ for some integer $n>0$. Most computations have been for the submodule category $\operatorname{sub}(\mathbb{Z}/(p^n))$. In fact, already in \cite{Bir35} a $1$-parameter family in $\operatorname{sub}(\mathbb{Z}/(p^6))$ of indecomposables $(f_x\colon M_1\xrightarrow{}M_2)$ depending on $p$ was constructed (see also \cite[Section 1.1]{Sch05}). The isomorphism class of $M_1$ and $M_2$ is constant in the family, and the number of indecomposables in the family goes to infinity as $p$ goes to infinity. This reflects the fact that $\operatorname{sub}(\mathbb{Z}/(p^6))$ is of infinite representation type, as shown in \cite{Pla76}. Determining the indecomposables and the AR-quiver of $\operatorname{sub}(\mathbb{Z}/(p^6))$ is still an open problem, called the Birkhoff problem after \cite{Bir35}.

The indecomposables of $\operatorname{sub}(\mathbb{Z}/(p^n))$ were computed in \cite{RW79} for $n\leq 4$, and for $n=5$ in \cite{RW99}. These are the representation finite cases. The category $\operatorname{sub}(\mathbb{Z}/(p^n))$ can also be determined from certain tiled orders over the $p$-adic integers, see \cite[Section 2]{Rum81}. Indecomposables for tiled orders were studied in \cite{Rum86,Rum90}.

\section{An epivalence and the Mimo construction}\label{Section:An epivalence and the Mimo construction}

Here we explain the construction of the epivalence from \cite{GKKP23}, and how the Mimo-construction introduced by Ringel and Schmidmeier gives an inverse to it. We also recall the application to commutative uniserial rings of length $3$ given in \cite{GKKP23}. Finally, we show how this result can be used to give a classification of finite and infinite representation type of monomorphism categories over $\mathbb{Z}/(p^n)$. 

\subsection{The epivalence}

Let $Q$ be a finite acyclic quiver and $\Lambda$ an Artin algebra. We first start by characterizing the \textit{injective} objects in $\operatorname{mono}(Q,\Lambda)$, i.e. the representations $M\in \operatorname{mono}(Q,\Lambda)$ satisfying  $\operatorname{Ext}^1_{\Lambda Q}(N,M)=0$ for all $N\in \operatorname{mono}(Q,\Lambda)$. To this end, we need to consider the restriction functor
\[
f^*\colon \operatorname{rep}(Q,\operatorname{mod}\Lambda)\to \prod_{\mathtt{i}\in Q_0}\operatorname{mod}\Lambda, \quad \quad (M_\mathtt{i},h_\alpha)_{\mathtt{i}\in Q_0,\alpha\in Q_1}\mapsto (M_\mathtt{i})_{\mathtt{i}\in Q_0}
\]
and its left adjoint
\[
f_!\colon \prod_{\mathtt{i}\in Q_0}\operatorname{mod}\Lambda\to \operatorname{rep}(Q,\operatorname{mod}\Lambda).
\]
For $\mathtt{i}\in Q_0$ let  $M(\mathtt{i})$ be the object in $\prod_{\mathtt{i}\in Q_0}\operatorname{mod}\Lambda$ given by $$M(\mathtt{i})_\mathtt{j}=
	\begin{cases}
	M, & \text{if}\ \mathtt{j}=\mathtt{i} \\
	0, & \text{if}\ \mathtt{j}\neq\mathtt{i}.
	\end{cases}$$
Applying $f_!$ to $M(\mathtt{i})$ gives the representation
\begin{equation*}
	f_!(M(\mathtt{i}))_\mathtt{k}=\bigoplus_{\substack{p\in Q_{\geq 0}\\ s(p)=\mathtt{i},t(p)=\mathtt{k}}}M \quad \text{and} \quad f_!(M(\mathtt{i}))_\alpha\colon \bigoplus_{\substack{p\in Q_{\geq 0}\\ s(p)=\mathtt{i},t(p)=\mathtt{k}}}M \to \bigoplus_{\substack{p\in Q_{\geq 0}\\ s(p)=\mathtt{i},t(p)=\mathtt{l}}}M
	\end{equation*}
 for a vertex $\mathtt{k}\in Q_0$ and an arrow $\alpha\colon \mathtt{k}\to \mathtt{l}$. Here $Q_{\geq 0}$ denotes the set of paths in $Q$. The map $f_!(M(\mathtt{i}))_\alpha$ is induced by the identity maps $M\xrightarrow{1} M$ between the components indexed by paths $p$ and $\alpha p$.
 
\begin{Proposition}\label{Proposition:InjectivesMono}
The following hold:
\begin{enumerate}
    \item\label{Proposition:InjectivesMono:1} A representation $M=(M_\mathtt{i},h_\alpha)\in \operatorname{mono}(Q,\Lambda)$ is injective if and only if $M_\mathtt{i}$ is an injective $\Lambda$-module for all $\mathtt{i}\in Q_0$.
    \item\label{Proposition:InjectivesMono:2} A representation $M=(M_\mathtt{i},h_\alpha)\in \operatorname{mono}(Q,\Lambda)$ is indecomposable injective if and only if it is isomorphic to a representation of the form $f_!(J(\mathtt{i}))$ where $\mathtt{i}\in Q_0$ and $J$ is an indecomposable injective $\Lambda$-module.
\end{enumerate}
\end{Proposition}

\begin{proof}
  See \cite[Proposition 2.4]{LZ13} or \cite[Corollary 3.8 and Section 4]{GKKP23}.
\end{proof}

Let $\overline{\operatorname{mono}}(Q,\Lambda)$ denote the injectively stable category of $\operatorname{mono}(Q,\Lambda)$. Explicitly, it has the same objects as $\operatorname{mono}(Q,\Lambda)$, and has morphism spaces
\[
\operatorname{Hom}_{\overline{\operatorname{mono}}(Q,\Lambda)}(M,N)= \operatorname{Hom}_{\operatorname{mono}(Q,\Lambda)}(M,N)/I(M,N)
\]
where $I(M,N)$ denotes the abelian groups of morphisms factoring through an injective object in $\operatorname{mono}(Q,\Lambda)$. The injectively stable category $\overline{\operatorname{mod}}\,\Lambda$ of $\operatorname{mod}\Lambda$ is defined similarly, using the injective $\Lambda$-modules. Finally, consider the category $\operatorname{rep}(Q,\overline{\operatorname{mod}}\,\Lambda)$ of representations of $Q$ in $\overline{\operatorname{mod}}\,\Lambda$. It is defined similarly as $\operatorname{rep}(Q,\operatorname{mod}\Lambda)$, but with components and morphisms in $\overline{\operatorname{mod}}\,\Lambda$ instead.

\begin{Definition}
    	An \textit{epivalence} is a full and dense functor which reflects isomorphisms (i.e. if $F(f)$ is an isomorphism, then $f$ must be an isomorphism).
\end{Definition}

\begin{Theorem}\label{Theorem:Epivalence}
    The composite
	\[
	\operatorname{mono}(Q,\Lambda)\to \operatorname{rep}(Q,\operatorname{mod}\Lambda)\to \operatorname{rep}(Q,\overline{\operatorname{mod}}\,\Lambda)
	\]
	induces an epivalence
	\[
	\Phi\colon \overline{\operatorname{mono}}(Q,\Lambda)\to \operatorname{rep}(Q,\overline{\operatorname{mod}}\,\Lambda).
	\]
	Furthermore, if $Q$ has at least one arrow, then $\Phi$ is an equivalence if and only if $\Lambda$ has global dimension $\leq 1$.
\end{Theorem}

\begin{proof}
    This follows from \cite[Theorem A]{GKKP23}.
\end{proof}

Theorem \ref{Theorem:Epivalence} is known for $\mathbb{A}_2$ by \cite[Corollary 5.4 and Lemma 5.5]{RS08b}, and for linearly oriented $\mathbb{A}_n$-quiver by \cite[Corollary 2.6 and Lemma 4.1]{XZZ14}. For $\Lambda=k[x]/(x^2)$ it recovers the homology functor for perfect differential $kQ$-modules in \cite[Theorem 1.1 b)]{RZ17}, and for $Q$ a linearly oriented $\mathbb{A}_n$-quiver and $\Lambda$ path algebra of linearly oriented $\mathbb{A}_m$-quiver it recovers (the dual of) the equivalence in \cite[Theorem 1.5]{BBOS20} as special cases. An analogue of Theorem \ref{Theorem:Epivalence} was recently shown for Gorenstein projective quiver representations over algebras whose stable Cohen-Macaulay Auslander algebra is semisimple \cite[Theorem 4.2]{HB24}.

Epivalences were introduced by Auslander in \cite[Chapter II]{Aus71} under the name representation equivalences, since they preserve and reflect important representation-theoretic properties. In particular, the epivalence in Theorem \ref{Theorem:Epivalence} induces a bijection
\begin{align}\label{BijectionIndec}
    \renewcommand{\arraystretch}{1}
		\renewcommand{\arraystretch}{1.8}
		\begin{array}{ccc}
		\renewcommand{\arraystretch}{1.1}
  \renewcommand{\arraystretch}{1.2}
		\begin{Bmatrix}
		\text{Indecomposable non-injective} \\
		\text{objects in $\operatorname{mono}(Q,\Lambda)$}
		\end{Bmatrix}^{\cong}
		&
		\xrightarrow{\cong}
		&
  \renewcommand{\arraystretch}{1.1}
		 \begin{Bmatrix}
		\text{Indecomposable objects} \\
		\text{in $\operatorname{rep} (Q,\overline{\operatorname{mod}}\, \Lambda)$}
		\end{Bmatrix}^{\cong}
		\end{array}
\end{align}
where $\cong$ indicates that we are considering isomorphism classes of representations. In many situations the right hand side of this bijection is easier to study than the left hand side. For example, if $\overline{\operatorname{mod}}\,\Lambda$ is itself a module category of an Artin algebra, then the left hand just consists of indecomposable modules over some Artin algebra. This holds true if $\Lambda$ is a path algebra of a linearly oriented $\mathbb{A}_n$-quiver, or a radical-square zero Nakayama algebra. In particular, for a commutative uniserial ring $\Lambda$ of length $2$, we have $\overline{\operatorname{mod}}\,\Lambda\cong \operatorname{mod}k$ where $k$ is the residue field of $\Lambda$. This gives
\begin{Theorem}\label{Theorem:BijectionUniserialLength2}
    Let $\Lambda$ be a commutative artinian uniserial ring of length $2$ with residue field $k$. Then the functor in Theorem \ref{Theorem:Epivalence} induces a bijection 
    \begin{align*}
    \renewcommand{\arraystretch}{1}
		\renewcommand{\arraystretch}{1.8}
		\begin{array}{ccc}
		\renewcommand{\arraystretch}{1.1}
  \renewcommand{\arraystretch}{1.2}
		\begin{Bmatrix}
		\text{Indecomposable non-injective} \\
		\text{objects in $\operatorname{mono}(Q,\Lambda)$}
		\end{Bmatrix}^{\cong}
		&
		\xrightarrow{\cong}
		&
  \renewcommand{\arraystretch}{1.1}
		 \begin{Bmatrix}
		\text{Indecomposable objects} \\
		\text{in $\operatorname{rep} (Q,\operatorname{mod}k)$}
		\end{Bmatrix}^{\cong}
		\end{array}
\end{align*}
In particular, $\operatorname{mono}(Q,\Lambda)$ is of finite representation type if and only if $Q$ is Dynkin.
\end{Theorem}
Note that this result also follows from \cite[Theorem 2.1]{Luo21}. It applies in particular to the dual numbers $k[x]/(x^2)$ and to $\mathbb{Z}/(p^2)$.

In order to recover results about indecomposables in $\operatorname{mono}(Q,\Lambda)$ from indecomposables in $\operatorname{rep} (Q,\overline{\operatorname{mod}}\, \Lambda)$, we need to construct an inverse to \eqref{BijectionIndec}. In the next subsection we explain how to do this, using the Mimo construction.

	\subsection{Mimo-construction}\label{Subsection:MimoConstruction}

The minimal right $\operatorname{mono}(Q,\Lambda)$-approximation in $\operatorname{rep}(Q,\operatorname{mod}\Lambda)$ is called the Mimo-construction. We start by giving a description of it.

 \begin{Definition}\label{Definition:Mimo}
     Let $(M_\mathtt{i},h_\alpha)$ be an object in $\operatorname{rep}(Q,\operatorname{mod}\Lambda)$. The \textit{Mimo construction} $$\operatorname{Mimo}(M_\mathtt{i},h_\alpha)=(M'_\mathtt{i},h'_\alpha)$$ of $(M_\mathtt{i},h_\alpha)$ is defined as follows:  
     \begin{enumerate}
         \item For each vertex $\mathtt{i}\in Q_0$ choose an injective envelope of $\Lambda$-modules $j_\mathtt{i}\colon K_\mathtt{i}\to J_\mathtt{i}$, where $K_\mathtt{i}$ is the kernel of the morphism
		\[
		 \bigoplus_{\substack{\alpha\in Q_1\\t(\alpha)=\mathtt{i}}} M_{s(\alpha)}\xrightarrow{(h_\alpha)_\alpha} M_\mathtt{i}.
		\]
  \item For each vertex $\mathtt{i}\in Q_0$ choose a lift $e_{\mathtt{i}}$ making the following diagram commutative
  \[
	\begin{tikzcd}
	K_\mathtt{i}\arrow{d}\arrow{dr}{j_\mathtt{i}} & \\
	\bigoplus_{\substack{\alpha\in Q_1\\t(\alpha)=\mathtt{i}}} M_{s(\alpha)} \arrow[r,dashed,"e_\mathtt{i}" below] & J_\mathtt{i}.
	\end{tikzcd}
	\]  
  \item For each vertex $\mathtt{i}\in Q_0$ let
		\[
		M'_{\mathtt{i}}= M_{\mathtt{i}}\oplus \bigoplus_{p\in Q_{\geq 0},t(p)=\mathtt{i}}J_{s(p)}
		\] 
		where $Q_{\geq 0}$ is the set of paths in $Q$, and $s(p)$ and $t(p)$ denotes the source and target of $p$, respectively.
  \item For each arrow $\beta\colon \mathtt{i}\to \mathtt{k}$ let
		\[
		h'_\beta\colon M_{\mathtt{i}}\oplus \bigoplus_{p\in Q_{\geq 0},t(p)=\mathtt{i}}J_{s(p)}\to M_{\mathtt{k}}\oplus \bigoplus_{q\in Q_{\geq 0},t(q)=\mathtt{k}}J_{s(q)}
		\]
	 be the morphism induced by the identity $J_{s(p)}\xrightarrow{1}J_{s(q)}$ for $q=\beta p$, the structure map $h_\beta\colon M_\mathtt{i}\to M_\mathtt{k}$, and the composite $M_\mathtt{i}\xrightarrow{} \bigoplus_{\substack{\alpha\in Q_1, t(\alpha)=\mathtt{k}}} M_{s(\alpha)}\xrightarrow{e_{\mathtt{k}}} J_{\mathtt{k}}$ where the first map is the canonical inclusion. 
     \end{enumerate}
 \end{Definition}

 \begin{Theorem}\label{Theorem:MinimalRightApproximation}
   Let $(M_\mathtt{i},h_\alpha)$ be an object in $\operatorname{rep}(Q,\operatorname{mod}\Lambda)$. The canonical projection
   \[
   \operatorname{Mimo}(M_\mathtt{i},h_\alpha)\to (M_\mathtt{i},h_\alpha), \quad \quad M_{\mathtt{i}}\oplus \bigoplus_{p\in Q_{\geq 0},t(p)=\mathtt{i}}J_{s(p)}\xrightarrow{\begin{pmatrix}
       1&0
   \end{pmatrix}} M_\mathtt{i}
   \]
   is a minimal right $\operatorname{mono}(Q,\Lambda)$-approximation.
 \end{Theorem}

 \begin{proof}
It is a right $\operatorname{mono}(Q,\Lambda)$-approximation by \cite[Lemma 3.2 and Proposition 3.3]{LZ13}. Minimality follows from \cite[Theorem 6.2]{GKKP23}.    
 \end{proof}
 The Mimo construction was first introduced in \cite{RS08} for the submodule category, where it was constructed by taking a minimal monomorphism, hence the name. In \cite{Zha11} the Mimo construction was extended to linearly oriented $\mathbb{A}_n$-quivers, and it was shown that it gives a minimal right $\operatorname{mono}(Q,\Lambda)$-approximation \cite[Theorem 2.1]{Zha11}. The general formula was obtained in \cite[Section 3a]{LZ13}. However, as far as we know minimality wasn't proven until \cite{GKKP23}. Note that there are shorter ways to show that $\operatorname{mono}(Q,\Lambda)$ has a minimal right approximation, see e.g. Subsection \ref{Subsection:Monomorphism categories over fields}. The advantage of Theorem \ref{Theorem:MinimalRightApproximation} is that it gives an explicit description of it. 

 Next we explain how the Mimo-construction provides an inverse to \eqref{BijectionIndec}. In the following, for $(M_\mathtt{i},h_\alpha)\in \operatorname{rep}(Q,\overline{\operatorname{mod}}\,\Lambda)$, we choose a representation $(\hat{M}_\mathtt{i},\hat{h}_\alpha)\in \operatorname{rep}(Q,\operatorname{mod}\Lambda)$ as follows:
 \begin{enumerate}
     \item The $\Lambda$-module $\hat{M}_\mathtt{i}$ is the unique module (up to isomorphy) with no nonzero injective summands, and which is isomorphic to $M_\mathtt{i}$ in $\overline{\operatorname{mod}}\,\Lambda$.
     \item The morphism $\hat{h}_\alpha\colon \hat{M}_{s(\alpha)}\to \hat{M}_{t(\alpha)}$ is a choice of a lift of $h_\alpha\colon  M_{s(\alpha)}\to M_{t(\alpha)}$ to $\operatorname{mod}\Lambda$. 
 \end{enumerate}
We get the following result.
 
 \begin{Theorem}\label{Theorem:MimoInverse}
   The association 
   \[
   (M_\mathtt{i},h_\alpha)\mapsto \operatorname{Mimo}(\hat{M}_\mathtt{i},\hat{h}_\alpha)
   \]
   is well-defined up to isomorphism and gives an inverse to the map \eqref{BijectionIndec}.
 \end{Theorem}

 \begin{proof}
     This follows from \cite[Proposition 7.8 and Theorem 7.9]{GKKP23}.
 \end{proof}
Theorem \ref{Theorem:MimoInverse} is known for $\mathbb{A}_2$ by \cite[Corollaries 5.3 and 5.4]{RS08}, and for linearly oriented $\mathbb{A}_n$ by \cite[Corollaries 2.5 and 2.6]{XZZ14}.

We apply Theorem \ref{Theorem:MimoInverse} to get the indecomposable monomorphic representations of a $\mathbb{D}_4$-quiver over $\mathbb{Z}/(p^2)$.

\begin{Example}\label{Example:D_n4OverCyclicLength2}
Let $k=\mathbb{Z}/(p)$ denote the finite field with $p$ elements, let $\Lambda=\mathbb{Z}/(p^2)$, and let $Q$ be the quiver 
\[
\begin{tikzcd}\mathtt{1}\arrow{rd}&\mathtt{2}\arrow{d}&\mathtt{3}\arrow{ld}\\&\mathtt{4}\end{tikzcd}
\]
It is well-known that the indecomposable objects in $\operatorname{mod} kQ$ are given by
\begin{align*}
\begin{tikzcd}[ampersand replacement=\&]k\arrow{rd}\&0\arrow{d}\&0\arrow{ld}\\{}\&0\end{tikzcd}\qquad
\begin{tikzcd}[ampersand
replacement=\&]0\arrow{rd}\&k\arrow{d}\&0\arrow{ld}\\{}\&0\end{tikzcd}\qquad 
\begin{tikzcd}[ampersand replacement=\&]0\arrow{rd}\&0\arrow{d}\&k\arrow{ld}\\{}\&0\end{tikzcd} \\
\begin{tikzcd}[ampersand replacement=\&]0\arrow{rd}\&0\arrow{d}\&0\arrow{ld}\\{}\&k\end{tikzcd}\qquad 
\begin{tikzcd}[ampersand replacement=\&]k\arrow{rd}{1}\&0\arrow{d}\&0\arrow{ld}\\{}\&k\end{tikzcd}\qquad
\begin{tikzcd}[ampersand
replacement=\&]0\arrow{rd}\&k\arrow{d}{1}\&0\arrow{ld}\\{}\&k\end{tikzcd}\\
\begin{tikzcd}[ampersand replacement=\&]0\arrow{rd}\&0\arrow{d}\&k\arrow{ld}{1}\\{}\&k\end{tikzcd}\qquad
\begin{tikzcd}[ampersand
replacement=\&]k\arrow{rd}[swap]{\left(\begin{smallmatrix}1\\0\end{smallmatrix}\right)}\&k\arrow{d}{\left(\begin{smallmatrix}0\\1\end{smallmatrix}\right)}\&k\arrow{ld}{\left(\begin{smallmatrix}1\\1\end{smallmatrix}\right)}\\{}\&k^2\end{tikzcd}\qquad 
\begin{tikzcd}[ampersand replacement=\&]0\arrow{rd}\&k\arrow{d}{1}\&k\arrow{ld}{1}\\{}\&k\end{tikzcd}\\
\begin{tikzcd}[ampersand
replacement=\&]k\arrow{rd}{1}\&0\arrow{d}\&k\arrow{ld}{1}\\{}\&k\end{tikzcd}\qquad 
\begin{tikzcd}[ampersand replacement=\&]k\arrow{rd}{1}\&k\arrow{d}{1}\&0\arrow{ld}\\{}\&k\end{tikzcd}\qquad
\begin{tikzcd}[ampersand
replacement=\&]k\arrow{rd}[swap]{1}\&k\arrow{d}{1}\&k\arrow{ld}{1}\\{}\&k\end{tikzcd}
\end{align*}
Applying the Mimo-construction and using Theorems \ref{Theorem:BijectionUniserialLength2} and \ref{Theorem:MimoInverse} we obtain the following list of indecomposable non-injective objects in $\operatorname{mono}(Q,\Lambda)$:
\begin{small}
\begin{align*}\begin{tikzcd}[ampersand replacement=\&]k\arrow{rd}{\iota}\&0\arrow{d}\&0\arrow{ld}\\{}\&\Lambda\end{tikzcd}\qquad
\begin{tikzcd}[ampersand
replacement=\&]0\arrow{rd}\&k\arrow{d}{\iota}\&0\arrow{ld}\\{}\&\Lambda\end{tikzcd}\qquad 
\begin{tikzcd}[ampersand replacement=\&]0\arrow{rd}\&0\arrow{d}\&k\arrow{ld}{\iota}\\{}\&\Lambda\end{tikzcd}\\
\begin{tikzcd}[ampersand replacement=\&]0\arrow{rd}\&0\arrow{d}\&0\arrow{ld}\\{}\&k\end{tikzcd}\qquad 
\begin{tikzcd}[ampersand replacement=\&]k\arrow{rd}{1}\&0\arrow{d}\&0\arrow{ld}\\{}\&k\end{tikzcd}\qquad
\begin{tikzcd}[ampersand
replacement=\&]0\arrow{rd}\&k\arrow{d}{1}\&0\arrow{ld}\\{}\&k\end{tikzcd}\\
\begin{tikzcd}[ampersand replacement=\&]0\arrow{rd}\&0\arrow{d}\&k\arrow{ld}{1}\\{}\&k\end{tikzcd}\qquad
\begin{tikzcd}[ampersand
replacement=\&]k\arrow{rd}[swap]{\left(\begin{smallmatrix}1\\0\\0\end{smallmatrix}\right)}\&k\arrow{d}{\left(\begin{smallmatrix}0\\1\\0\end{smallmatrix}\right)}\&k\arrow{ld}{\left(\begin{smallmatrix}1\\1\\\iota\end{smallmatrix}\right)}\\{}\&k^2\oplus \Lambda\end{tikzcd}\qquad 
\begin{tikzcd}[ampersand replacement=\&]0\arrow{rd}\&k\arrow{d}{\left(\begin{smallmatrix}1\\0\end{smallmatrix}\right)}\&k\arrow{ld}{\left(\begin{smallmatrix}1\\\iota\end{smallmatrix}\right)}\\{}\&k\oplus \Lambda\end{tikzcd}\\
\begin{tikzcd}[ampersand
replacement=\&]k\arrow{rd}{\left(\begin{smallmatrix}1\\0\end{smallmatrix}\right)}\&0\arrow{d}\&k\arrow{ld}{\left(\begin{smallmatrix}1\\\iota\end{smallmatrix}\right)}\\{}\&k\oplus \Lambda\end{tikzcd}\qquad 
\begin{tikzcd}[ampersand replacement=\&]k\arrow{rd}[swap]{\left(\begin{smallmatrix}1\\0\end{smallmatrix}\right)}\&k\arrow{d}{\left(\begin{smallmatrix}1\\\iota\end{smallmatrix}\right)}\&0\arrow{ld}\\{}\&k\oplus \Lambda\end{tikzcd}\qquad
\begin{tikzcd}[ampersand
replacement=\&]k\arrow{rd}[swap]{\left(\begin{smallmatrix}1\\0\\0\end{smallmatrix}\right)}\&k\arrow{d}{\left(\begin{smallmatrix}1\\\iota\\0\end{smallmatrix}\right)}\&k\arrow{ld}{\left(\begin{smallmatrix}1\\0\\\iota\end{smallmatrix}\right)}\\{}\&k\oplus \Lambda^2\end{tikzcd}
\end{align*}
\end{small}
Here $\iota$ denotes the canonical inclusion. Finally, by Proposition \ref{Proposition:InjectivesMono} \eqref{Proposition:InjectivesMono:2} the four indecomposable injective objects in $\operatorname{mono}(Q,\Lambda)$ are given by
\begin{align*}
f_!(\Lambda(\mathtt{1}))&=\begin{tikzcd}[ampersand replacement=\&]\Lambda\arrow{rd}{1}\&0\arrow{d}\&0\arrow{ld}\\{}\&\Lambda\end{tikzcd}\qquad
f_!(\Lambda(\mathtt{2}))=\begin{tikzcd}[ampersand replacement=\&]0\arrow{rd}\&\Lambda\arrow{d}{1}\&0\arrow{ld}\\{}\&\Lambda\end{tikzcd}\\
f_!(\Lambda(\mathtt{3}))&=\begin{tikzcd}[ampersand replacement=\&]0\arrow{rd}\&0\arrow{d}\&\Lambda\arrow{ld}{1}\\{}\&\Lambda\end{tikzcd}\qquad
f_!(\Lambda(\mathtt{4}))=\begin{tikzcd}[ampersand replacement=\&]0\arrow{rd}\&0\arrow{d}\&0\arrow{ld}\\{}\&\Lambda\end{tikzcd}\\
\end{align*}
\end{Example}
	\subsection{Stable equivalences}

 Theorem \ref{Theorem:Epivalence} tells us that a large part of the representation theory of $\operatorname{mono}(Q,\Lambda)$ only depends on the injectively stable category $\overline{\operatorname{mod}}\,\Lambda$ of $\Lambda$. In particular
 
 \begin{Corollary}\label{Corollary:StableEquivalence}
     Let $Q$ be a finite acyclic quiver and let $\Lambda$ and $\Gamma$ be Artin algebras. Assume we have an equivalence
     \[
     \overline{\operatorname{mod}}\,\Lambda\cong  \overline{\operatorname{mod}}\,\Gamma
     \]
     between the injectively stable categories. Then there is a bijection
     \begin{align*}
    \renewcommand{\arraystretch}{1}
		\renewcommand{\arraystretch}{1.8}
		\begin{array}{ccc}
		\renewcommand{\arraystretch}{1.1}
  \renewcommand{\arraystretch}{1.2}
		\begin{Bmatrix}
		\text{Indecomposable non-injective} \\
		\text{objects in $\operatorname{mono}(Q,\Lambda)$}
		\end{Bmatrix}^{\cong}
		&
		\xrightarrow{\cong}
		&
  \renewcommand{\arraystretch}{1.1}
		 \begin{Bmatrix}
		\text{Indecomposable non-injective} \\
		\text{objects in $\operatorname{mono}(Q,\Gamma)$}
		\end{Bmatrix}^{\cong}
		\end{array}
\end{align*}
 \end{Corollary}

If $\Lambda$ and $\Gamma$ are commutative uniserial rings of length $3$ with the same residue field, then we have an equivalence
\[
\overline{\operatorname{mod}}\,\Lambda \cong \overline{\operatorname{mod}}\, \Gamma
\]
see \cite[Proposition 8.11]{GKKP23}. In particular, Corollary \ref{Corollary:StableEquivalence} can be applied to this setting. In fact, in this case it can be extended to a bijection between all indecomposable objects. Furthermore, the bijection preserves the underlying \textit{partition vectors} (sometimes called the \textit{type}). We explain what this means below.

Let $\Lambda$ be a commutative uniserial ring of length $n$. Then there is a bijection between isomorphism classes of finitely generated $\Lambda$-modules, and partitions, i.e.  sequences $$\overline{\alpha}=(\alpha_1\geq \alpha \geq \ldots \geq \alpha_m\geq 1)$$ with $\alpha_1\leq n$. Explicitly the bijection sends $\overline{\alpha}$ to the $\Lambda$-module
\[
M(\overline{\alpha})\colonequals \bigoplus_{i=1}^m M(\alpha_i)
\]
where $M(\alpha_i)$ is the unique (up to isomorphism) indecomposable $\Lambda$-module with length $\alpha_i$. Given a representation $(M_\mathtt{i},h_\alpha)\in \operatorname{rep}(Q,\operatorname{mod}\Lambda)$, we have an associated partition $\overline{\alpha}^\mathtt{i}$  for each $\mathtt{i}\in Q_0$, defined by $M(\overline{\alpha}^\mathtt{i})\cong M_\mathtt{i}$. The tuple $(\overline{\alpha}^\mathtt{i})_{\mathtt{i}\in Q_0}$ is called the \textit{partition vector} of $(M_\mathtt{i},h_\alpha)$.

\begin{Theorem}{\cite[Theorem 8.13]{GKKP23}}\label{Theorem:BijectionUniserialLength3}
Let $\Lambda$ and $\Gamma$ be commutative uniserial rings of length $\leq 3$ with the same residue field. Then there is a bijection
     \begin{align*}
    \renewcommand{\arraystretch}{1}
		\renewcommand{\arraystretch}{1.8}
		\begin{array}{ccc}
		\renewcommand{\arraystretch}{1.1}
  \renewcommand{\arraystretch}{1.2}
		\begin{Bmatrix}
		\text{Indecomposable objects} \\
		\text{in $\operatorname{mono}(Q,\Lambda)$}
		\end{Bmatrix}^{\cong}
		&
		\xrightarrow{\cong}
		&
  \renewcommand{\arraystretch}{1.1}
		 \begin{Bmatrix}
		\text{Indecomposable objects} \\
		\text{in $\operatorname{mono}(Q,\Gamma)$}
		\end{Bmatrix}^{\cong}
		\end{array}
\end{align*}
which preserves the underlying partition vector of the representations.
\end{Theorem}

 \subsection{Representation type over cyclic abelian groups}

The monomorphism category is said to be of finite or infinite representation type, respectively, if it has finitely or infinitely many indecomposable representations up to isomorphism. In this subsection we classify when the monomorphism category over a cyclic abelian group is of finite or infinite representation type. This builds on \cite{GKKP23} and previous work by other authors. 

We start by investigating the behavior of representation type under field extensions.
\begin{Proposition}\label{Proposition:ExtensionOfFields}
    Let $Q$ be a finite acyclic quiver, $\Lambda$ a finite-dimensional $k$-algebra over a field $k$, and $\ell$ an algebraic and separable extension of $k$. Assume $\operatorname{mono}(Q,\Lambda)$ is of finite representation type. The following hold:
    \begin{enumerate}
        \item\label{Proposition:ExtensionOfFields:1} Any representation $M$ in $\operatorname{mono}(Q,\Lambda\otimes_k \ell)$ is a direct summand of a representation of the form $N\otimes_k\ell$ where $N\in \operatorname{mono}(Q,\Lambda)$.
        \item\label{Proposition:ExtensionOfFields:2} $\operatorname{mono}(Q,\Lambda\otimes_k \ell)$ is of finite representation type.
    \end{enumerate}
\end{Proposition}

\begin{proof}
 
 Note that for any $N\in \operatorname{mono}(Q,\Lambda)$ and any field extension $K$ over $k$, the representation $N\otimes_{k}K$ must be finite-dimensional over $K$, and hence has finitely many non-isomorphic indecomposable summands as an object in $\operatorname{mono}(Q,\Lambda\otimes_k K)$. Therefore, the category $\operatorname{mono}(Q,\Lambda\otimes_k K)$ is of finite representation type if $K$ satisfies condition \eqref{Proposition:ExtensionOfFields:1} in the proposition, since $\operatorname{mono}(Q,\Lambda)$ is of finite representation type. Hence, if we set
 \[
 \mathcal{S}\colonequals\{\text{intermediate fields }\ell' \text{ of }\ell/k\mid \ell' \text{ satisfies condition \eqref{Proposition:ExtensionOfFields:1}} \text{ in the proposition} \}.
 \] then we only need to show that $\ell\in \mathcal{S}$.
 
 We first show that $\mathcal{S}$ has a maximal element, using Zorn's lemma. Consider a totally ordered set $\mathcal{P}$ and a subset $(\ell_i)_{i\in \mathcal{P}}\subseteq \mathcal{S}$ such that $i\leq j$ if and only if $\ell_i\subseteq \ell_j$. We want to find an element in $\mathcal{S}$ containing all the $\ell_i$'s. We claim that the field $$\ell'=\bigcup_{i\in \mathcal{P}}\ell_i$$ satisfies this. It clearly contains all the $\ell_i$'s, so we only need to show that $\ell'\in \mathcal{S}$. 
 
 Let $\Gamma$ be a finite-dimensional $k$-algebra. We first prove that any finite-dimensional $\Gamma\otimes_k\ell'$-module is of the form $N\otimes_{\ell_i}\ell'$ for some $i\in \mathcal{P}$ and some finite-dimensional $\Gamma\otimes_k\ell_i$-module $N$. Indeed, a finite-dimensional $\Gamma\otimes_k\ell'$-module can be represented as a $k$-algebra morphism $\Gamma\to M_n(\ell')$ for some integer $n$, where $M_n(\ell')$ denotes the set of $n\times n$-matrices with coefficients in $\ell'$. Now since $\ell'=\bigcup_{i\in \mathcal{P}}\ell_i$, we have that $$M_n(\ell')=\bigcup_{i\in \mathcal{P}}M_n(\ell_i).$$ Since $\Gamma$ is finite-dimensional, the image of the morphism $\Gamma\to M_n(\ell')$ must land in some $M_n(\ell_i)$. This is equivalent to the module being of the form $N\otimes_{\ell_i}\ell'$. 

 Applying this to $\Gamma=\Lambda Q$, we get that any representation in $\operatorname{mono}(Q,\Lambda\otimes_k \ell')$ is of the form $N'\otimes_{\ell_i}\ell'$ for some $i\in \mathcal{P}$ and some $N'$ in $\operatorname{mono}(Q,\Lambda\otimes_k \ell_i)$. Since $\ell_i\in \mathcal{S}$, the representation $N'$ must be a summand of  $N\otimes_k\ell_i$ for some $N\in \operatorname{mono}(Q,\Lambda)$. Hence $N'\otimes_{\ell_i}\ell'$ is a summand of $(N\otimes_k\ell_i)\otimes_{\ell_i}\ell'\cong N\otimes_k\ell'$. This shows that $\ell'\in \mathcal{S}$.

 Now by Zorn's lemma, $\mathcal{S}$ must have a maximal element, say $\ell_0$. If $\ell_0\neq \ell$, choose an element $x\in \ell\backslash\ell_0$, and consider the field extension $\ell_0(x)/\ell_0$. We want to show that $\ell_0(x)\in \mathcal{S}$, contradicting the maximality of $\ell_0$. First note that the extension $\ell_0(x)/\ell_0$ is finite and separable since $\ell/k$ is separable and algebraic. Now let $M\in \operatorname{mono}(Q,\Lambda\otimes_k\ell_0(x))$ be arbitrary, and consider the composite morphism $$M\otimes_{\ell_0}\ell_0(x)\xrightarrow{\cong} M\otimes_{\ell_0(x)}(\ell_0(x)\otimes_{\ell_0} \ell_0(x))\xrightarrow{1\otimes \phi} M\otimes_{\ell_0(x)}\ell_0(x)\cong M$$
 in $\operatorname{mono}(Q,\Lambda\otimes_k\ell_0(x))$. Here $\phi\colon \ell_0(x)\otimes_{\ell_0} \ell_0(x)\to \ell_0(x)$ is the $\ell_0(x)\otimes_{\ell_0} \ell_0(x)$-bimodule morphism given by multiplication in $\ell_0(x)$. Since $\ell_0(x)$ is a finite and separable extension of $\ell_0$, the map $\phi$ is split as a morphism of $\ell_0(x)\otimes_{\ell_0} \ell_0(x)$-bimodules. Hence, $M$ is a summand of $M\otimes_{\ell_0}\ell_0(x)$ in $\operatorname{mono}(Q,\Lambda\otimes_k\ell_0(x))$. Now let ${}_{\ell_0}|M$ be $M$ considered as a representation over $\Lambda\otimes_k\ell_0$. Then ${}_{\ell_0}|M$ must be finite-dimensional, since $\ell_0(x)$ is finite-dimensional over $\ell_0$. Hence, ${}_{\ell_0}|M$ must lie in $\operatorname{mono}(Q,\Lambda\otimes_k\ell_0)$. Since $\ell_0\in \mathcal{S}$, it follows that ${}_{\ell_0}|M$ must be a summand of $N\otimes_k\ell_0$ for some $N\in \operatorname{mono}(Q,\Lambda)$. Therefore $M\otimes_{\ell_0}\ell_0(x)$ must be a summand of $N\otimes_k\ell_0(x)$. Hence, $M$ itself is a direct summand of $N\otimes_k\ell_0(x)$. This shows that $M$ satisfies \eqref{Proposition:ExtensionOfFields:1}, and since $M$ was arbitrary it follows that $\ell_0(x)\in \mathcal{S}$. This contradicts the maximality of $\ell_0$. Hence, we must have that $\ell_0=\ell$, which proves the claim. 
\end{proof}

\begin{Remark}
Note that Proposition \ref{Proposition:ExtensionOfFields} is not true if the extension is not separable, even if $Q$ just consists of one vertex and no arrows. Indeed, let $k$ be a field of characteristic $p$ for some prime $p$, and let $\alpha$ be an element in $k$ which is not a $p$th power. Then $x^p-\alpha$ is an irreducible polynomial, so $\ell=k[x]/(x^p-\alpha)$ is a field extension of $k$, see \cite[Lemma 4.4]{Jac85}. Let $\beta$ be a $p$th root of $\alpha$ in $\ell$, so $\beta^p=\alpha$. We set $\Lambda=\ell\otimes_k\ell$. Now we have that 
\[
\ell\otimes_k\ell \cong \ell[x]/(x^p-\alpha)=\ell[x]/(x-\beta)^p\cong \ell[y]/y^p.
\]
Hence it is a nilpotent uniserial ring of length $p$. Therefore
\[
\ell\otimes_k\ell\otimes_k\ell\cong(\ell\otimes_k\ell)\otimes_\ell (\ell\otimes_k\ell)\cong \ell[y]/y^p\otimes_\ell \ell[z]/z^p\cong \ell[y,z]/(y^p,z^p)
\]
which is of infinite representation type for $p\geq 2$. Since $\ell\otimes_k\ell\cong \ell[y]/y^p$ is of finite representation type, we get that $\ell[y]/y^p$ considered as a $k$-algebra gives a counterexample to Proposition \ref{Proposition:ExtensionOfFields}.
\end{Remark}

We also need the following result on reflection functors.

\begin{Theorem}\label{Theorem:SinkSourceReflection}
   Let $Q$ be a finite acyclic quiver and $\Lambda$ a selfinjective algebra. Let $Q(v)$ be the quiver obtained from $Q$ by a sink or  source reflection at a vertex $v$. Then there is an equivalence 
   \[
   \overline{\operatorname{mono}}(Q,\Lambda) \xrightarrow{\cong} \overline{\operatorname{mono}}(Q(v),\Lambda)
   \]
between the injectively stable categories.
\end{Theorem}
\begin{proof}
    See \cite[Theorem 1.1]{LS22}. 
\end{proof}

We can now prove the classification result.

\begin{Theorem}\label{Theorem:RepTypeUniserial}
    Let $Q$ be a finite acyclic quiver. Then $\operatorname{mono}(Q,\mathbb{Z}/(p^n))$ is of finite representation type if and only if 
    \begin{enumerate}
    \item\label{Theorem:RepTypeUniserial:1} $n=2$ and $Q$ is Dynkin, or
    \item\label{Theorem:RepTypeUniserial:2} The underlying graph of $Q$ is $\mathbb{A}_m$ and $$(n,m)\in \{(3,2),(4,2),(5,2),(3,3),(3,4)\}$$
    \end{enumerate}
\end{Theorem}

\begin{proof}
If $n=2$, then we can use Theorem \ref{Theorem:BijectionUniserialLength2} to get that $\operatorname{mono}(Q,\mathbb{Z}/(p^2))$ is of finite representation type if and only if $Q$ is Dynkin. 

Assume $n\geq 3$. The epimorphism $\mathbb{Z}/(p^n)\to \mathbb{Z}/(p^2)$ induces a natural inclusion $$\operatorname{mono}(Q,\mathbb{Z}/(p^2))\to \operatorname{mono}(Q,\mathbb{Z}/(p^n))$$ which preserves indecomposables. Hence, $Q$ must be Dynkin if $\operatorname{mono}(Q,\mathbb{Z}/(p^n))$ is of finite representation type. Assume the underlying graph of $Q$ is Dynkin of type $\mathbb{D}_m$ or $\mathbb{E}_6,\mathbb{E}_7,\mathbb{E}_8$. By Theorem \ref{Theorem:BijectionUniserialLength3} there is a bijection between the indecomposables in  $\operatorname{mono}(Q,\mathbb{Z}/(p^3))$ and in $\operatorname{mono}(Q,k[x]/(x^3))$, where $k=\mathbb{Z}/(p)$ is the finite field with $p$ elements. By \cite[Theorem 4.6]{Lu20} the category $\operatorname{mono}(Q,\overline{k}[x]/(x^3))$ is of infinite representation type, where $\overline{k}$ is the algebraic closure of $k$. Since $k$ is a finite field, it is perfect, so $\overline{k}$ is a separable extension of $k$. Hence, by Proposition \ref{Proposition:ExtensionOfFields} \eqref{Proposition:ExtensionOfFields:2} the category $\operatorname{mono}(Q,k[x]/(x^3))$ must also be of infinite representation type. Therefore, $\operatorname{mono}(Q,\mathbb{Z}/(p^3))$ is of infinite representation type, so $\operatorname{mono}(Q,\mathbb{Z}/(p^n))$ is of infinite representation type.

Now assume the underlying graph of $Q$ is Dynkin of type $\mathbb{A}_n$. Since $\mathbb{Z}/(p^n)$ is commutative uniserial, it must be selfinjective. Hence, the representation type of $\operatorname{mono}(Q,\mathbb{Z}/(p^n))$ does not depend on the orientation of $Q$ by Theorem \ref{Theorem:SinkSourceReflection}. We can therefore assume $Q$ is a linearly oriented $\mathbb{A}_m$-quiver. Then by \cite{Pla76} (see also \cite[Theorem 1.3]{Sim02}) the category $\operatorname{mono}(Q,\mathbb{Z}/(p^n))$ is of finite representation type if and only if $(n,m)\in\{(3,2),(4,2),(5,2),(3,3),(3,4)\}$.  
\end{proof}

	\section{Valuated groups}\label{Section:ValuatedGroups}
	
	\subsection{Valuated groups and embeddings of subgroups}

 The height of an element plays an important role in abelian group theory, for example in Ulm's and Prüfer theorems on classification of certain infinite abelian groups. The data of the height can be recovered from the data of the $p$\textit{-height} for all primes $p$. The $p$-height $h(x)$ of an element $x$ in an abelian group $G$ is the largest integer $n\geq 0$ such that $p^ny=x$ for an element $y$ in $G$, and $\infty$ otherwise. 
 
 Let $G$ be a bounded abelian $p$-group $G$, i.e. each element have order $p^n$ for some $n$. Let $H$ be a subgroup of $G$. The $p$-height function $h|_H$ of $G$ restricted to $H$ provides an important invariant of the embedding of $H$ into $G$. In fact, it determines left minimal embeddings up to isomorphism, see \cite[Theorem 1.2]{HRW84}. The $p$-valuation of an abelian group, introduced in \cite{RW79}, is an axiomatization of the properties of $h|_H$. In this subsection we review the theory of $p$-valuations, following \cite{RW79}. We also explain how the category of $p$-valuated groups is related to the submodule category $\operatorname{Sub}(\mathbb{Z}/(p^n))$. 
	\begin{enumerate}
		\item  A $p$\textit{-valuation} on an abelian group $B$ is a function $$v\colon B\to \mathbb{Z}_{\geq 0}\cup \{\infty\}$$ satisfying for all $x,y\in B$:
		\begin{enumerate}
			\item $v(x+y)\geq \operatorname{min}\{v(x),v(y)\}$ .
			\item $v(px)>v(x)$ (where we assume $\infty>\infty$ and $\infty>n$ for all $n\in \mathbb{Z}_{\geq 0}$).
		\end{enumerate}
		\item A $p$-\textit{valuated} abelian group is an abelian group $B$ equipped with a $p$-valuation $v\colon B\to \mathbb{Z}_{\geq 0}\cup \{\infty\}$. 
		\item  A morphism $\phi\colon B\to B'$ of $p$-valuated abelian groups is a morphism of abelian groups satisfying $v(x)\leq v(\phi(x))$ for all $x\in B$.
  \item We let $\mathcal{V}_p^n$ denote the category with objects $p$-valuated abelian groups $B$ satisfying $v(x)\leq n-1$ for all nonzero $x\in B$, and with morphisms as above.
	\end{enumerate}	
 Since $p\cdot 0=0$ and $v(p\cdot x)>v(x)$, we must have that $v(0)=\infty$. Note that $0$ is the only element with valuation $\geq n$ in a $p$-valuated group $B$ in $\mathcal{V}_p^n$. Hence, all elements of $B$ must be $p^n$-torsion, i.e. $p^n\cdot B=0$. In particular, $B$ is $p$-local in the sense of \cite{RW79}, i.e. multiplication by a prime different from $p$ is an automorphism of $B$.
 
	We have the following alternative description for $\mathcal{V}_p^n$, which is used in \cite{RW99}. 

\begin{Lemma}\label{Lemma:AltDescriptionValuatedGroup}
	The category $\mathcal{V}_p^n$ is isomorphic to the category where
 \begin{itemize} 
 \item The objects are given by increasing sequence $(B(i))_{0\leq i\leq n}$ of abelian groups
			\[
			0=B(n)\subseteq B(n-1)\subseteq \cdots \subseteq B(1)\subseteq B(0)
			\]
			satisfying $p\cdot B(i)\subseteq B(i+1)$ for $0\leq i\leq n-1$.
   \item The morphisms are given by tuples $(\phi(i)\colon B(i)\to B'(i))_{0\leq i\leq n}$ of morphisms of abelian groups making the diagrams
			\[
			\begin{tikzcd}
			B(i)\arrow{r}{\phi(i)}\arrow{d}{\subseteq} &B'(i)\arrow{d}{\subseteq}\\
			B(i-1)\arrow{r}{\phi(i-1)} &B'(i-1)
			\end{tikzcd}
			\]
			commutative for all $1\leq i\leq n-1$. 
   \end{itemize}
	\end{Lemma}

 \begin{proof}
 Given an object $B$ in $\mathcal{V}_p^n$, we get an increasing sequence of abelian groups
		\[
		0=B(n)\subseteq B(n-1)\subseteq \cdots \subseteq B(1)\subseteq B(0)
		\]
as in the lemma by setting $B(i)\colonequals \{x\mid v(x)\geq i\}.$ 
 Given a morphism $\phi\colon B\to B'$ in $\mathcal{V}_p^n$, we get a tuple $(\phi(i)\colon B(i)\to B'(i))_{0\leq i\leq n-1}$ of morphisms as in the lemma by letting $\phi(i)$ be the restriction of $\phi$ to $B(i)$. Note that this is well-defined since $v(x)\leq v(\phi(x))$ for all $x\in B$. Hence, we have a functor from $\mathcal{V}_p^n$ to the category described in the lemma. We claim that this is an isomorphism. Indeed, its inverse sends an increasing sequence $(B(i))_{0\leq i\leq n}$ of abelian groups to the abelian group $B(0)$ with $p$-valuation given by $v(0)=\infty$ and $v(x)=i$ if $0\neq x\in B(i)$ and $x\notin B(i+1)$.
	\end{proof}
From now on we will use the two descriptions of $\mathcal{V}_p^n$ interchangeably.	
	
	Recall that a morphism $f$ is called a kernel (resp. cokernel) if there exists a morphism $g$ such that $f$ is the kernel of $g$ (resp. cokernel of $g$).
	A \textit{quasi-abelian category} is an additive category with kernels and cokernels, such that %the pushout of a kernel along any morphism is still a kernel, and such that the pullback of a cokernel along any morphism is still a cokernel. This is equivalent to requiring that all 
	its kernel-cokernel pairs form an exact structure in the sense of Quillen, see \cite[Section 4]{Bue10} for more details. We refer to \cite{Bue10} or \cite[Appendix A]{Kel90} for an introduction to exact categories. %For more details on the history of quasi-abelian categories see Section 2 in \cite{Rum08}.
	
	\begin{Theorem}\label{Theorem:ExactStructureValuatedGroups}
		The category $\mathcal{V}_p^n$ is quasi-abelian. Furthermore:
		\begin{enumerate}
			\item\label{Theorem:ExactStructureValuatedGroups:1} If $\phi\colon B\to B'$ is a morphism in $\mathcal{V}_p^n$, then the following are equivalent:
			\begin{enumerate}
				\item\label{Theorem:ExactStructureValuatedGroups:1a} $\phi$ is a kernel.
				\item\label{Theorem:ExactStructureValuatedGroups:1b} $\phi$ is injective and $v(\phi(x))=v(x)$ for all $x\in B$.
				\item\label{Theorem:ExactStructureValuatedGroups:1c} $\phi$ is injective and \begin{equation}\label{Equation:CartesianSquare}
				\begin{tikzcd}
				B(i)\arrow{r}{\phi(i)}\arrow{d}{} &B'(i)\arrow{d}{}\\
				B(0)\arrow{r}{\phi(0)} &B'(0)
				\end{tikzcd}
				\end{equation}  
				is a pullback square for each $0\leq i\leq n$. 
			\end{enumerate}
			\item\label{Theorem:ExactStructureValuatedGroups:2} If $\phi\colon B\to B'$ is a morphism in $\mathcal{V}_p^n$, then the following are equivalent:
			\begin{enumerate}
				\item\label{Theorem:ExactStructureValuatedGroups:2a} $\phi$ is a cokernel.
				\item\label{Theorem:ExactStructureValuatedGroups:2b} $\phi$ is surjective and for any $x'\in B'$ we can find $x\in B$ satisfying $$\phi(x)=x' \quad \text{and} \quad v(x)=v(x').$$
\item\label{Theorem:ExactStructureValuatedGroups:2c} $\phi(i)\colon B(i)\to B'(i)$ is surjective for each $0\leq i\leq n$.
			\end{enumerate}
			\item\label{Theorem:ExactStructureValuatedGroups:3} A sequence $B\xrightarrow{\phi}B'\xrightarrow{\psi}B''$ in $\mathcal{V}_p^n$ is a kernel-cokernel pair if and only if $$0\to B(i)\xrightarrow{\phi(i)}B'(i)\xrightarrow{\psi(i)}B''(i)\to 0$$ is a short exact sequence of abelian groups for all $0\leq i\leq n$.
		\end{enumerate}
	\end{Theorem}
	
	\begin{proof}
		By \cite[Theorem 3]{RW79} the category $\mathcal{V}_p^n$ is preabelian, i.e. has kernels and cokernels. Following \cite{RW77}, a kernel in a preabelian category is called semi-stable if its pushout along any morphism is still a kernel, and a cokernel is called semi-stable if its pullback along any morphism is still a cokernel. Furthermore, a semi-stable kernel is called stable if its cokernel is semi-stable, and a semi-stable cokernel is called stable if its kernel is semi-stable.  By \cite{SW11} any preabelian category has a maximal exact structure where the inflations are the stable kernels and the deflations are the stable cokernels. In particular this holds true for $\mathcal{V}_p^n$. Since the stable cokernels are precisely the cokernels of the stable kernels, it suffices to show that all kernels in $\mathcal{V}_p^n$ are stable in order to prove that $\mathcal{V}_p^n$ is quasi-abelian.
		
		By \cite[Theorem 6]{RW79} the stable kernels in $\mathcal{V}_p^n$ are the morphisms $\phi\colon B\to B'$ which are injective and satisfy $v(\phi(x))=v(x)$ for all $x\in B$. The latter condition corresponds precisely to the equality $B(i)=B(0)\cap B'(i)$ for each $0\leq i\leq n$, i.e. that the diagrams in \eqref{Theorem:ExactStructureValuatedGroups:1} are pullback squares. This shows that the morphisms described in \eqref{Theorem:ExactStructureValuatedGroups:1b} and \eqref{Theorem:ExactStructureValuatedGroups:1c} are the same, and that they coincide with the stable kernels. Finally, any kernel in $\mathcal{V}_p^n$ must satisfy \eqref{Theorem:ExactStructureValuatedGroups:1b} by the description in \cite[Theorem 3]{RW79}.  This shows \eqref{Theorem:ExactStructureValuatedGroups:1} and that $\mathcal{V}_p^n$ is quasi-abelian.
		
		The equivalence between \eqref{Theorem:ExactStructureValuatedGroups:2a} and \eqref{Theorem:ExactStructureValuatedGroups:2b} follows immediately from the characterization of kernels and cokernels in \cite[Theorem 3]{RW79}. It can also easily be checked that \eqref{Theorem:ExactStructureValuatedGroups:2b} and \eqref{Theorem:ExactStructureValuatedGroups:2c} are equivalent. This proves \eqref{Theorem:ExactStructureValuatedGroups:2}.
		
		For \eqref{Theorem:ExactStructureValuatedGroups:3}, assume we have a sequence $$B\xrightarrow{\phi}B'\xrightarrow{\psi}B''$$
		of $p$-valuated abelian groups in $\mathcal{V}_p^n$. This is a kernel-cokernel pair if and only if $\psi$ is a cokernel in $\mathcal{V}_p^n$ and $\phi$ is the kernel of $\psi$. Now from \eqref{Theorem:ExactStructureValuatedGroups:2} we know that $\psi$ is a cokernel if and only if $\psi(i)$ is surjective for all $0\leq i\leq n$, and by \cite[Theorem 3]{RW79} we know that $\phi$ is the kernel of $\psi$ if and only if $\phi(i)$ is the kernel of $\psi(i)$ for each $0\leq i\leq n$. Hence $(\psi,\phi)$ form a kernel-cokernel pair in $\mathcal{V}_p^n$ if and only if
		$$0\to B(i)\xrightarrow{\phi(i)}B'(i)\xrightarrow{\psi(i)}B''(i)\to 0$$ 
		is a short exact sequence of abelian groups for all $0\leq i\leq n$. This proves \eqref{Theorem:ExactStructureValuatedGroups:3}.
	\end{proof}
	
	\begin{Remark}
		Note that $\mathcal{V}_p^n$ is not in general abelian, as described in the paragraph after Theorem 3 in \cite{RW79}. To see this, use the description of $\mathcal{V}_p^n$ in Lemma \ref{Lemma:AltDescriptionValuatedGroup}. Consider the objects $B=(B(i))_{0\leq i\leq n}$ and $B=(B(i))_{0\leq i\leq n}$ given by
  \begin{align*}
    & B(0)=B(1)=\mathbb{Z}/(p) \quad \text{and} \quad B(i)=0 \text{ for }2\leq i\leq n \\
  &  B'(0)=\mathbb{Z}/(p) \quad \text{and} \quad B'(i)=0 \text{ for }1\leq i\leq n
  \end{align*}
  so we get increasing sequences
  \begin{align*}
  & B= (0\subseteq 0\subseteq\cdots \subseteq 0\subseteq \mathbb{Z}/(p)\subseteq \mathbb{Z}/(p)), \quad B'=(0\subseteq 0\subseteq\cdots \subseteq 0\subseteq 0\subseteq\mathbb{Z}/(p))
   \end{align*}
  of abelian groups. Let $\phi\colon B'\to B$ be the morphism given by $\phi(0)=\operatorname{id}_{\mathbb{Z}/(p)}$ and $\phi(i)=0$ for $1\leq i\leq n$. Then it follows from the description in \cite[Theorem 3]{RW79} that $\phi$ has zero kernel and cokernel. But $\phi$ is not an isomorphism, so $\mathcal{V}_p^n$ can't be abelian 
  %let $G$ a $p$-torsion abelian group, e. and consider the canonical morphism $\phi\colon B\to B'$ where 
		%\begin{align*}
		%	& B(0)=B(1)=G \quad \text{and} \quad B(i)=0 \text{ for }2\leq i\leq n-1. \\
		%	& B'(0)=G \quad \text{and} \quad B'(i)=0 \text{ for }1\leq i\leq n-1.
		%\end{align*}
	%	Then it follows from the description in \cite[Theorem 3]{RW79} that $\phi$ has zero kernel and cokernel. But $\phi$ is not an isomorphism, so $\mathcal{V}_p^n$ can't be abelian.
\end{Remark}
	
	\begin{Remark}
	Note that the category $\mathcal{V}_p$ of $p$-local valuated abelian groups studied in \cite{RW79} is not quasi-abelian. Indeed, in \cite{RW79} they show that what they call embeddings are the kernels in $\mathcal{V}_p$, that the embeddings they call nice are the stable kernels in $\mathcal{V}_p$ \cite[Theorem 6]{RW79}, and that there exists embeddings which are not nice (see last paragraph of Section 3 in \cite{RW79}). Note that embeddings and nice embeddings are the same in $\mathcal{V}_p^n$.
	\end{Remark}

	Next we recall the characterization of injective objects in $\mathcal{V}_p^n$ from \cite{RW79}. Following the terminology of \cite[Appendix A]{Kel90}, we call a kernel in $\mathcal{V}_p^n$ for an \textit{inflation}, a cokernel for a \textit{deflation}, and a kernel-cokernel pair for a \textit{conflation}. Recall that the $p$-height $h(x)$ of an element $x$ in $B$ is the largest integer $n\geq 0$ such that there exists $y\in B$ satisfying $p^ny=x$. Otherwise we say the height is $\infty$. Note that $h(0)=\infty$
	\begin{Proposition}\label{ValuatedGroupsEnoughInjectives}
		As an exact category, $\mathcal{V}_p^n$ has enough injectives. Furthermore, the following are equivalent for an object $B$ in $\mathcal{V}_p^n$:
		\begin{enumerate}
			\item\label{ValuatedGroupsEnoughInjectives:1} $B$ is injective in the exact structure of $\mathcal{V}_p^n$.
			\item\label{ValuatedGroupsEnoughInjectives:2} $v(x)=h(x)$ for all $x\in B$.
			\item\label{ValuatedGroupsEnoughInjectives:3} $B(i)=p^i\cdot B$ for each $0\leq i\leq n$.
		\end{enumerate}
	\end{Proposition}
	
	\begin{proof}
		Clearly \eqref{ValuatedGroupsEnoughInjectives:2} and \eqref{ValuatedGroupsEnoughInjectives:3} are equivalent. In \cite{RW79} they call a $p$-valuated group satisfying the condition $v(x)=h(x)$ for all $x\in B$ for a \textit{group}. Assume $B$ satisfies this condition. Since the underlying abelian group of $B$ is $p^n$-torsion, it must be algebraically compact. Therefore, by \cite[Theorem 9]{RW79} any inflation $B\to B''$ from $B$ must be split. This shows that any object satisfying \eqref{ValuatedGroupsEnoughInjectives:2} must be injective in $\mathcal{V}_p^n$. Since any object $B'$ in $\mathcal{V}_p^n$ has an inflation into a group (in the sense above)  by \cite[Theorem 2]{RW79}, this shows that $\mathcal{V}_p^n$ has enough injectives. Finally, any injective object in $\mathcal{V}_p^n$ must have an inflation into a group, and by injectivity it must be a split monomorphism. Therefore any injective object is summand of a group. But clearly any summand of a group is still a group. This implies that any injective object is a group, which proves the claim. 
	\end{proof}
	
	We now  explain how $\mathcal{V}_p^n$ and $\operatorname{Sub}(\mathbb{Z}/(p^n))$ are related. Recall that $\operatorname{Sub}(\mathbb{Z}/(p^n))$ consists of all monomorphic representation of $\mathtt{1}\to \mathtt{2}$ over $\mathbb{Z}/(p^n)$, not just the finitely generated ones.  Define a functor \begin{equation}\label{FunctorPhi}
	\Phi\colon \operatorname{Sub}(\mathbb{Z}/(p^n))\to \mathcal{V}_p^n
	\end{equation}
	as follows:
	\begin{itemize}
		\item For each object $M=(M_1\subseteq M_2)$ in $\operatorname{Sub}(\mathbb{Z}/(p^n))$, let $\Phi(M)$ be the abelian group $M_1$ with the $p$-valuation given by the restriction of the $p$-height of $M_2$, i.e. 
  $$v(x)=\operatorname{sup}\{n\geq 0\mid \text{there exists }y\in M_2 \text{ with }p^ny=x\}$$
  for $x\in M_1$.
		\item For each morphism in $\operatorname{Sub}(\mathbb{Z}/(p^n))$ $$g=(g_1,g_2)\colon (M_1\subseteq M_2)\to (N_1\subseteq N_2)$$ let $\Phi(g)\colon M_1\to N_1$ be the map $g_1$ considered as a morphism in $\mathcal{V}_p^n$.
	\end{itemize}
Since the diagram
 \[
			\begin{tikzcd}
			M_1\arrow{r}{g_1}\arrow{d}{\subseteq} &N_1\arrow{d}{\subseteq}\\
			M_2\arrow{r}{g_2} &N_2
			\end{tikzcd}
			\]
   is commutative, $g_1\colon M_1\to N_1$ is a morphism of $p$-valuated abelian groups. Hence, $\Phi$ is a well-defined functor. In the notation of Lemma \ref{Lemma:AltDescriptionValuatedGroup} we have $$\Phi(M)(i)=M_1\cap p^i\cdot M_2.$$
	Now let $\mathcal{X}$ be the full subcategory of $\operatorname{Sub}(\mathbb{Z}/(p^n))$ consisting of all objects isomorphic to objects of the form $0\to M$ with $M\in \operatorname{Mod}\mathbb{Z}/(p^n)$. Note that $\mathcal{X}$ is closed under direct sums and summands. We can define a new category $\operatorname{Sub}(\mathbb{Z}/(p^n))/\mathcal{X}$ whose objects are the same as the objects in $\operatorname{Sub}(\mathbb{Z}/(p^n))$, and whose morphism spaces are given by
	\[
	\operatorname{Hom}_{\operatorname{Sub}(\mathbb{Z}/(p^n))/\mathcal{X}}(M,N)=\frac{\operatorname{Hom}_{\operatorname{Sub}(\mathbb{Z}/(p^n))}(M,N)}{\mathcal{X}(M,N)}
	\]
	where $\mathcal{X}(M,N)$ is the ideal of all morphisms $M\to N$ factoring through an object in $\mathcal{X}$. Since $\Phi$ vanishes on objects in $\mathcal{X}$, it induces a functor \begin{equation}\label{EquivalenceMonoValuated}
	\operatorname{Sub}(\mathbb{Z}/(p^n))/\mathcal{X}\to \mathcal{V}_p^n.
	\end{equation}

	\begin{Theorem}\label{Theorem:EquivalenceMonoValuatedGroups}
		The functor \eqref{EquivalenceMonoValuated} is an equivalence. 
	\end{Theorem}
	
	\begin{proof}
		It suffices to show that $\Phi$ is full and dense, and that it sends a morphism to $0$ if and only if the morphism factor through an object in $\mathcal{X}$. 
		
		Let $$g=(g_1,g_2)\colon (M_1\subseteq M_2)\to (N_1\subseteq N_2)$$ be a morphism in $\operatorname{Sub}(\mathbb{Z}/(p^n))$, and assume $\Phi(g)$ is $0$. By construction this implies $g_1$ is $0$. But then we can write $g=h\circ k$ where 
		\[k=(0,g_2)\colon (M_1\subseteq M_2)\to (0\subseteq N_2) \quad \text{and} \quad h=(0,1_{N_2})\colon (0\subseteq N_2)\to (N_1\subseteq N_2).
		\]
		This shows that $g$ factors through an object in $\mathcal{X}$. Conversely, if $g$ factors through an object in $\mathcal{X}$, then clearly $\Phi(g)=0$.
		
		Next we show that $\Phi$ is dense. Let $B$ be an object in $\mathcal{V}_p^n$. Using Proposition \ref{ValuatedGroupsEnoughInjectives} we can  find an inflation $B\to B'$ in $\mathcal{V}_p^n$ where $B'$ satisfies $B'(i)=p^i\cdot B'(0)$ for all $0\leq i\leq n$. Since $B\to B'$ is an inflation, we get pullback diagrams as in Theorem \ref{Theorem:ExactStructureValuatedGroups} \eqref{Theorem:ExactStructureValuatedGroups:1}. In other words, we have equalities $B(i)=B(0)\cap p^i\cdot B'(0)$ for $0\leq i\leq n$. Hence, $\Phi$ applied to $B(0)\subseteq B'(0)$ must be isomorphic to $B$, so $\Phi$ is dense.
		
		Finally, we show that $\Phi$ is full. Let $B=\Phi(M_1\subseteq M_2)$ and $B'=\Phi(N_1\subseteq N_2)$, and let $\phi\colon B\to B'$ be a morphism in $\mathcal{V}_p^n$. Define $C$ and $C'$ to be the $p$-valuated abelian groups given by $C(i)=p^i\cdot M_2$ and $C'(i)=p^i\cdot N_2$. Then $C$ and $C'$ are injective in $\mathcal{V}_p^n$, and there exists canonical inflations $B\to C$ and $B'\to C'$ induced from the inclusions $M_1\subseteq M_2$ and $N_1\subseteq N_2$. By injectivity, the composite 
		\[
		B\xrightarrow{\phi}B'\to C'
		\]
		can be lifted via the inflation $B\to C$ to a morphism $\psi\colon C\to C'$. In particular, we have a commutative diagram
		\[
		\begin{tikzcd}
		B(0)\arrow{r}{\phi(0)}\arrow{d}{} &B'(0)\arrow{d}{}\\
		C(0)\arrow{r}{\psi(0)} &C'(0)
		\end{tikzcd}
		\]
		Since $B(0)=M_1$ and $C(0)=M_2$ and $B'(0)=N_1$ and $C'(0)=N_2$, we get a morphism $g=(\phi(0),\psi(0))\colon (M_1\subseteq M_2)\to (N_1\subseteq N_2)$ in $\operatorname{Sub}(\mathbb{Z}/(p^n))$ such that $\Phi(g)=\phi$. This shows that $\Phi$ is full.
	\end{proof}
	
	\begin{Remark}
		By \cite[Theorem 1.2]{HRW84} the category $\mathcal{V}_p^n$ is equivalent to a category $P$ whose objects are the same as in $\operatorname{Sub}(\mathbb{Z}/(p^n))$, and where a morphism $(M_1\subseteq M_2)\to (N_1\subseteq N_2)$ is given by a morphism $M_1\to N_1$ of $\mathbb{Z}/(p^n)$-modules, which can be extended to a morphism $M_2\to N_2$. From Theorem \ref{Theorem:EquivalenceMonoValuatedGroups} we see that $\operatorname{Sub}(\mathbb{Z}/(p^n))/\mathcal{X}$ is equivalent to $P$.  
	\end{Remark}

 Let $\vartheta_p^n$ denote the full subcategory of $\mathcal{V}_p^n$ consisting of all finitely generated $p$-valuated abelian groups $B$. 

 \begin{Proposition}
      The subcategory $\vartheta_p^n$ is closed under kernels, cokernels, and extensions in $\mathcal{V}_p^n$. Hence, it is a quasi-abelian category. As a quasi-abelian category, it has enough injectives, and an object in $\vartheta_p^n$ is injective if and only if it is injective in $\mathcal{V}_p^n$.
 \end{Proposition}

 \begin{proof}
   It is clear that $\vartheta_p^n$ is closed under kernels, cokernels, and extensions, and so is a quasi-abelian category. Also by the construction in \cite[Theorem 1]{RW79} any object in $\vartheta_p^n$ has an inclusion into an object in $\vartheta_p^n$ which is injective in $\mathcal{V}_p^n$. Since the inclusion $\vartheta_p^n\to \mathcal{V}_p^n$ is exact, it must also be injective in $\vartheta_p^n$. Hence, $\vartheta_p^n$ has enough injectives, and an object in $\vartheta_p^n$ is injective if and only if it is injective in $\mathcal{V}_p^n$.
 \end{proof}

 Recall that $\operatorname{sub}(\mathbb{Z}/(p^n))$ consists of all representations $M_1\subseteq M_2$ where $M_1,M_2$ are finitely generated $\mathbb{Z}/(p^n)$-modules. Let $\mathcal{Y}$ denote the subcategory $\mathcal{X}\cap \operatorname{sub}(\mathbb{Z}/(p^n))$ of $\operatorname{sub}(\mathbb{Z}/(p^n))$. Then \eqref{EquivalenceMonoValuated}
   restricts to a functor
   \begin{equation}\label{EquivalenceMonoValuatedFinGen}
	\operatorname{sub}(\mathbb{Z}/(p^n))/\mathcal{Y}\to \vartheta_p^n.
	\end{equation} 

 \begin{Theorem}\label{Theorem:EquivalenceMonoValuatedGroupsFinGen}
The functor \eqref{EquivalenceMonoValuatedFinGen} is an equivalence.
 \end{Theorem}

 \begin{proof}
    Fully faithfulness of \eqref{EquivalenceMonoValuatedFinGen} follows immediately from \eqref{EquivalenceMonoValuated} being fully faithful. The fact that \eqref{EquivalenceMonoValuatedFinGen} is dense follows from the same argument as in the proof of Theorem \ref{Theorem:EquivalenceMonoValuatedGroups}.
 \end{proof}

	Theorem \ref{Theorem:EquivalenceMonoValuatedGroupsFinGen} implies that there is a bijection between isomorphism classes of indecomposables in $\vartheta_p^n$ and indecomposables in $\operatorname{sub}(\mathbb{Z}/(p^n))$ which are not contained in $\mathcal{Y}$. The main result in \cite{RW99} gives a description of the indecomposables in $\vartheta_p^5$, which can  be used to determine the indecomposables in  $\operatorname{sub}(\mathbb{Z}/(p^5))$. We give more details in the next two subsections.

	\subsection{Valuated trees}\label{Subsection:ValuatedTrees} 
	
	In this subsection we recall the definition of valuated trees, how they describe $p$-valuated abelian groups, following \cite{HRW77}, and how to recover indecomposable representations in $\operatorname{sub}(\mathbb{Z}/(p^n))$ from irretractable valuated trees. The following definitions are taken from \cite{HRW77}. 
	\begin{enumerate}
		\item A \textit{tree} is a set $T$ with a distinguished element $*\in T$ and a function $p\colon T\to T$ satisfying
		\begin{enumerate}
			\item $p(*)=*$.
			\item For all $x\in T$ there exists $n\in \mathbb{Z}_{\geq 0}$ such that $p^n(x)=*$.
		\end{enumerate} 
		
		\item Let $T$ be a tree. A \textit{valuation} on $T$ is a function 
		\[
		v\colon T\to \mathbb{Z}_{\geq 0}\cup \{\infty\} 
		\]
		satisfying $v(p(x))>v(x)$ for all $x\in T$, where we assume $\infty>\infty$. A \textit{valuated tree} is a tree together with a valuation on it. It is called \textit{finite} if $T$ is a finite set.
		\item  A morphism $f\colon T\to T'$ of valuated trees is a map of sets satisfying 
		\begin{enumerate}
			\item $f(p(x))=p(f(x))$ for all $x\in T$.
			\item $v(f(x))\geq v(x)$ for all $x\in T$.
		\end{enumerate} 
		\item $\operatorname{ValTree}_n$ is the category whose objects are finite valuated trees $T$ where $v(x)\leq n-1$ for all $*\neq x\in T$, and whose morphisms are morphisms of valuated trees.
	\end{enumerate}
	Let $T\in \operatorname{ValTree}_n$. Since $p(*)=*$ and $v(p(x))>v(x)$, we must have that $v(*)=\infty$. Also, since $v(x)\leq n-1$ for all $*\neq x\in T$, we must have that $p^n(x)=*$ for all $x\in T$. 
	
	Following \cite{HRW77}, we depict a valuated tree $T$ as a graph where the nodes are elements of $T\backslash\{*\}$. There is an edge from $x$ to $p(x)$ going downwards if $p(x)\neq *$. Each node $x$ also has an entry from $\mathbb{Z}_{\geq 0}\cup \{\infty\}$, corresponding to $v(x)$. 
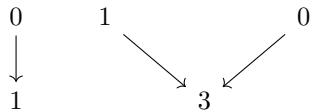
	\[
	\begin{tikzcd}
	0 \arrow{d}  \\
	1   
	\end{tikzcd} \quad \quad
	\begin{tikzcd}
	1 \arrow{rd} && 0 \arrow{ld} \\
	& 3  & 
	\end{tikzcd} 
	\]
	A finite valuated tree $T$ can also be described as a finite sequences of numbers (corresponding to the values $v(x)$ for $*\neq x\in T$) together with a parenthesization, see for example \cite{RW99}. This can be done recursively as follows:
	\begin{enumerate}
		\item If $T$ has only one element $x\neq *$, then the sequence only consists of the number $v(x)$.
		\item For a general $T$, let $T'=\{x_1,\dots x_k\}$ be the elements of $T\backslash\{*\}$ satisfying $p(x)=*$, and let 
		\[
		T_i=\{x\in T\mid p^m(x)=x_i \text{ for some }m\geq 0\}
		\]
		be the set of elements lying above $x_i$. Let $N_i$ be the parenthesized sequences of numbers associated to the tree obtained from $T_i$ by identifying $x_i$ with $*$. Then the parenthesized sequences of numbers associated to $T$ is 
		\[
		(v(x_1)N_1)\dots (v(x_k)N_k).
		\]
	\end{enumerate}
	For example, for the trees depicted above (from left to right) we get the sequences 
	\[
	 10 \quad \quad 3(1)(0).
	\]
	Note that the order is not unique. Indeed, the second tree can also be depicted by the sequence $3(0)(1)$.
	
	Similarly to $p$-valuated abelian groups, we can describe objects in $\operatorname{ValTree}_n$ in terms of finite filtrations.

\begin{Lemma}\label{Lemma:AltDescriptionValuatedTree}
	$\operatorname{ValTree}_n$ is isomorphic to the category where
 \begin{itemize} 
 \item Objects are increasing sequence $(T(i))_{0\leq i\leq n}$ of sets
			\[
			\{*\}=T(n)\subseteq T(n-1)\subseteq \cdots \subseteq T(1)\subseteq T(0)
			\]
			together with a map $p\colon T(0)\to T(0)$ satisfying $p(T(i))\subseteq T(i+1)$ for $0\leq i\leq n-1$.
   \item Morphisms are tuples $(f(i)\colon T(i)\to T'(i))_{0\leq i\leq n}$ of maps of sets for which
			\[
			\begin{tikzcd}
			T(i)\arrow{r}{f(i)}\arrow{d}{\subseteq} &T'(i)\arrow{d}{\subseteq}\\
			T(i-1)\arrow{r}{f(i-1)} &T'(i-1)
			\end{tikzcd} \quad \text{and} \quad 
			\begin{tikzcd}
			T(i)\arrow{r}{f(i)}&T'(i)\\
			T(i-1)\arrow{r}{f(i-1)} \arrow{u}{p} &T'(i-1) \arrow{u}{p}
			\end{tikzcd}
			\]
			are commutative for all $1\leq i\leq n$ .  
   \end{itemize}
	\end{Lemma}

 \begin{proof}
 The proof is identical to the one of Lemma \ref{Lemma:AltDescriptionValuatedGroup}.
\end{proof}
	
	To an object $B$ in $\vartheta_n^p$ we can associate a valuated tree whose underlying set and valuation are the ones of $B$, where $0\in B$ gets identified with the distinguished element $*$, and where the map $B\to B$ is given by multiplication by the prime $p$. This gives a forgetful functor $\vartheta_n^p\to \operatorname{ValTree}_n$.  The following proposition describes its left adjoint. The description is taken from Section $2$ in \cite{HRW77}.

 \begin{Lemma}\label{Lemma:UniqueSumS(X)}
     Given $T\in\operatorname{ValTree}_n$, let $A_T$ be the free abelian group on $T$, modulo the relations
			\begin{align*}
				&p\cdot [x]- [p(x)] \quad \text{for all }x\in T \\
				& [*].
			\end{align*} Then any element $b$ in $A_T$ can be written in a unique way as a sum $\sum_{x\in T}b_x[x]$ where $b_*=0$ and $0\leq b_x\leq p-1$ for all $x\in T$.
 \end{Lemma}

 \begin{proof}
     By definition, we can present an element $b$ in $A_T$ as a sum $\sum_{x\in T}b'_x[x]$ with $b'_x\in \mathbb{Z}$ for all $x\in T$. By repeatedly using the relations $p\cdot [x]- [p(x)]$ we can assume $b'_x\geq 0$ for all $x\in T$. Choose such a presentation $\sum_{x\in T}b_x[x]$ where $\sum_{x\in T}b_x$ is minimal. Then $0\leq b_x\leq p-1$ for all $x\in T$, since otherwise we could use the relation $p\cdot[x]- [p(x)]$ to make the sum $\sum_{x\in T}b_x$ smaller. Similarly, $b_*=0$ since otherwise we can use the second relation to make the sum smaller. Now for $*\neq x\in T$ we have a well-defined morphism $A_T\to \mathbb{Z}/(p)$ sending an element $\sum_{x\in T}b'_x[x]$ in $A_T$ to the residue of $b'_x$ in $\mathbb{Z}/(p)$, and this is independent of choice of presentation $\sum_{x\in T}b'_x[x]$ in $A_T$. Hence, if $b=\sum_{x\in T}b_x[x]=\sum_{x\in T}b'_x[x]$, then $b_x=b'_x$ in $\mathbb{Z}/(p)$, and so if $0\leq b_x,b'_x\leq p-1$, then we must have that $b_x=b'_x$ in $\mathbb{Z}$. This shows uniqueness.
 \end{proof}
 
\begin{Proposition}
 The forgetful functor $\vartheta_n^p\to \operatorname{ValTree}_n$ has a left adjoint 
		\[
		S\colon \operatorname{ValTree}_n\to  \vartheta_n^p
		\]
		defined as follows:
		\begin{enumerate}
			\item Let $T$ be an object in $\operatorname{ValTree}_n$. Then the underlying group of $S(T)$ is the group $A_T$ in Lemma \ref{Lemma:UniqueSumS(X)}. The $p$-valuation of $0\neq b\in S(T)$ is given by
			\[
			v(b)\colonequals \operatorname{min}\{v(x)\mid b_x\neq 0\}.
			\] 
   where $b=\sum_{x\in T}b_x[x]$ is a sum as in Lemma \ref{Lemma:UniqueSumS(X)}.
			\item Let $f\colon T\to T'$ be a morphism in $\operatorname{ValTree}_n$. Then
			\[
			S(f)(b)=\sum_{x\in T}b_x[f(x)].
			\]
   where $b=\sum_{x\in T}b_x[x]$ in $S(T)$ is a sum as in Lemma \ref{Lemma:UniqueSumS(X)}.
		\end{enumerate}  
	\end{Proposition}
	
	\begin{proof}
		It is clear that $S$ as defined gives a functor $S\colon \operatorname{ValTree}_n\to  \vartheta_n^p$. Now let $f\colon T\to B$ be a morphism of valuated trees with $T\in \operatorname{ValTree}_n$ and $B\in \vartheta_p^n$. Define 
		\[
		\overline{f}\colon S(T)\to B
		\] 
		where if $b=\sum_{x\in T}b_x[x]$ for $0\leq b_x\leq p-1$, then $\overline{f}(b)=\sum_{x\in T}b_xf(x)$. Then $\overline{f}$ is the unique extension of  $f$ to a morphism $S(T)\to B$ of $p$-valuated abelian groups. This implies that $S$ is left adjoint to the forgetful functor, which proves the claim.
	\end{proof}	
	
	A morphism $r\colon T\to T$ of valuated trees is a \textit{retraction} if it satisfies $r=r^2$. A valuated tree $T$ is called \textit{irretractable} if the only retractions $r\colon T\to T$ are the identity and the morphism sending everything to the distinguished element $*$.
	
	\begin{Theorem}\label{Theorem:IrretractableTreesValuatedGroups}
		Let $T \in \operatorname{ValTree}_n$. The following hold:
		\begin{enumerate}
			\item\label{Theorem:IrretractableTreesValuatedGroups:1}   $S(T)$ is indecomposable if and only if $T$ is irretractable.
			\item\label{Theorem:IrretractableTreesValuatedGroups:2} Assume $T$ is irretractable, and let $T'\in \operatorname{ValTree}_n$. If $S(T)$ is isomorphic to $S(T')$, then $T$ is isomorphic to $T'$.
			\item \label{Theorem:IrretractableTreesValuatedGroups:3} 
			There exists a collection $T_1,\dots T_m\in \operatorname{ValTree}_n$ of irretractable valuated trees, unique up to isomorphism and permutation, such that 
			\[
			S(T)\cong S(T_1)\oplus \dots S(T_m).
			\]
		\end{enumerate}
	\end{Theorem}
	
	\begin{proof}
		If $r\colon T\to T$ is a nontrivial retraction, then $S(r)\colon S(T)\to S(T)$ is a nontrivial retraction, so $S(T)$ is decomposable. Hence, $S(T)$ being indecomposable implies $T$ is irretractable. The converse follows from \cite[Theorem 7]{HRW77}. Part \eqref{Theorem:IrretractableTreesValuatedGroups:2} follows from \cite[Theorem 6]{HRW77}, and part \eqref{Theorem:IrretractableTreesValuatedGroups:3} follows from \cite[Theorem 8]{HRW77}.
	\end{proof}
	
	Combining Theorem \ref{Theorem:IrretractableTreesValuatedGroups} with the equivalence in Theorem \ref{Theorem:EquivalenceMonoValuatedGroupsFinGen} tells us that each irretractable valuated tree in $\operatorname{ValTree}_n$ corresponds to an indecomposable representation in $\operatorname{sub}(\mathbb{Z}/(p^n))$. However, we don't get any information on how the representation looks. In the next part we give an explicit description of it.
	
	Recall that a morphism $f\colon M\to N$ of $\mathbb{Z}/(p^n)$-modules is \textit{left minimal} if any morphism $g\colon N\to N$ satisfying $g\circ f=f$ must be an isomorphism. By \cite[Theorem 2.2]{ARS95} this is equivalent to requiring that if $N\cong N'\oplus N''$ such that $f$ corresponds to a morphism 
	\[
	\begin{bmatrix}f'\\0 \end{bmatrix}\colon M\to N'\oplus N''
	\]
	then $N''=0$. In other words, $f$ has no nonzero summands in $\mathcal{Y}$ when we consider it as an object in $\operatorname{rep}(\mathtt{1}\to \mathtt{2},\operatorname{mod}\mathbb{Z}/(p^n))$. Given $T\in \operatorname{ValTree}_n$ we want to construct an object in $\operatorname{sub}(\mathbb{Z}/(p^n))$ which is left minimal, and which is isomorphic to $S(T)$ under the functor $\Phi$ in \eqref{FunctorPhi}. The main idea of the construction is taken from \cite{Ric88}. 
	
	Let $T\in \operatorname{ValTree}_n$. We construct a filtration 
	\[
	T_{0}\subseteq T_1\subseteq \cdots \subseteq T_{n-1}
	\]
	of valuated trees as follows: First set $T_{0}\colonequals T$. Then define $T_i$ for $i\geq 1$ recursively as
	\[
	T_{i}=T_{i-1}\cup \{(x,m)\in T\times \mathbb{Z}_{\geq 0}\mid v(x)=i \text{ and }h(x)<i\text{ and }0\leq m<i   \}
	\]
	where $h(x)=\operatorname{max}\{t\geq 0\mid \text{there exists }y\in T_{i-1} \text{ with }p^t(y)=x\}$ is the $p$-height of $x$. We extend the map $p\colon T_{i-1}\to T_{i-1}$ to a map $p\colon T_i\to T_i$ by setting 
	$$p(x,m)=
	\begin{cases}
	(x,m+1), & \text{if}\ m+1<i \\
	x, & \text{if}\ m+1=i.
	\end{cases}$$
	We extend the valuation on $T_{i-1}$ to a valuation on $T_i$ by setting $v(x,m)=m$. Intuitively, we are adding elements to $T_{i-1}$ so that every element of value $i$ also has height $i$. 
\begin{Example}
    Given the tree
    \[
   T=
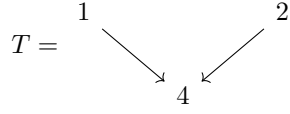
\begin{tikzcd}
	1 \arrow{rd} && 2 \arrow{ld} \\
	& 4  &
    \end{tikzcd}
    \]
    the construction above applied to $T$ gives the filtration $T_0\subseteq T_1\subseteq T_2\subseteq T_3\subseteq T_4$ where $T_0=T$ and $T_2=T_3$, and where $T_1,T_2,T_4$ are 
    
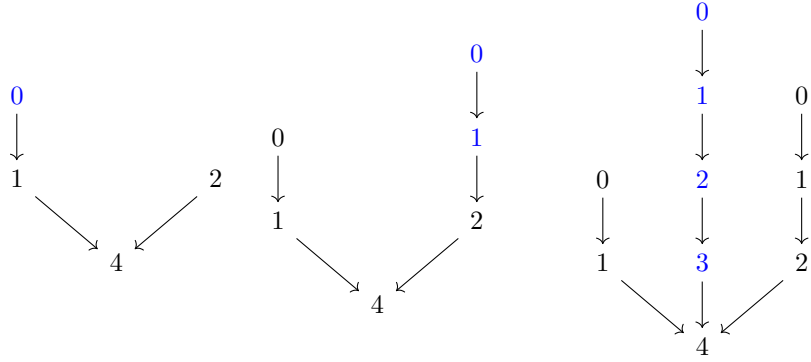
\[
    \begin{tikzcd}\color{blue}0\arrow{d} \\
	1 \arrow{rd} && 2 \arrow{ld} \\
	& 4 
    \end{tikzcd} \quad
    \begin{tikzcd}&& \color{blue}0\arrow{d}& \\
    0\arrow{d} && \color{blue}1\arrow{d}  \\
	1 \arrow{rd} &  & 2 \arrow{ld}  &  \\
	& 4   
 \end{tikzcd} \quad \begin{tikzcd}& \color{blue}0\arrow{d}\\
 & \color{blue}1\arrow{d} & 0\arrow{d}& \\
    0\arrow{d} & \color{blue}2\arrow{d} & 1\arrow{d}  \\
	1 \arrow{rd} & \color{blue}3\arrow{d} & 2 \arrow{ld}  &  \\
	& 4   
 \end{tikzcd}
    \]
   respectively. Here the new vertices are written in blue.
\end{Example}

Fix $T\in \operatorname{ValTree}_n$ and let $T_0\subseteq T_1\subseteq \cdots \subseteq T_{n-1}$ be a filtration as above. Let $\gamma\colon T\to T_{n-1}$ denote the inclusion of valuated trees.  	
	\begin{Theorem}\label{Theorem:ValuatedTreeToMonicRep}
		 The following hold:
		\begin{enumerate}
			\item\label{Theorem:ValuatedTreeToMonicRep:1} $S(T_{n-1})$ is an injective object in $\vartheta_p^n$.
			\item\label{Theorem:ValuatedTreeToMonicRep:2} $S(\gamma)\colon S(T)\to S(T_{n-1})$ is an inflation in $\vartheta_p^n$.
			\item\label{Theorem:ValuatedTreeToMonicRep:3} $S(\gamma)\colon S(T)\to S(T_{n-1})$ is a left minimal morphism of $\mathbb{Z}/(p^n)$-modules.       
		\end{enumerate}
	\end{Theorem}
	
	\begin{proof}
		By construction of $T_i$, we have that all elements $x\in T_i$ satisfying $v(x)\leq i$ must also satisfy $v(x)=h(x)$. In particular, $v(x)=h(x)$ for all $*\neq x\in T_{n-1}$. This implies that $v(x)=h(x)$ for any $x\in S(T_{n-1})$. Hence, $S(T_{n-1})$ is an injective object in $\vartheta_p^n$ by Proposition \ref{ValuatedGroupsEnoughInjectives}, which proves \eqref{Theorem:ValuatedTreeToMonicRep:1}. 
		
		Next we observe that the inclusion $\gamma\colon T\to T_{n-1}$ satisfies $v(\gamma(x))=v(x)$ for all $x\in T$ by construction. Hence, the same holds for $S(\gamma)\colon S(T)\to S(T_{n-1})$. By Theorem \ref{Theorem:ExactStructureValuatedGroups} \eqref{Theorem:ExactStructureValuatedGroups:1} this implies that $S(\gamma)$ is an inflation, which proves \eqref{Theorem:ValuatedTreeToMonicRep:2}.
		
		To prove \eqref{Theorem:ValuatedTreeToMonicRep:3}, assume $S(T_{n-1})\cong A\oplus A'$ as $\mathbb{Z}/(p^n)$-module, so that $S(\gamma)$ is of the form
		\[
		\begin{bmatrix}
		f\\0
		\end{bmatrix}\colon S(T)\to A\oplus A'.
		\]
		If $A'\neq 0$, we can choose a nonzero element $a'\in A'$ satisfying $p\cdot a'=0$. By Lemma \ref{Lemma:UniqueSumS(X)} we can write
		\begin{align*}
			a'=\sum_{x\in T_{n-1}}a'_xx 
		\end{align*}
		for a tuple of integers $(a'_x)_{x\in T_{n-1}}$ satisfying $0\leq a'_x\leq p-1$. Since $a'\neq 0$ and $a'\notin S(T)$, there exists $z\in T_{n-1}\backslash T$ for which $a'_z\neq 0$. Choose such an element $z$ of minimal height. By construction it must be of the form $(y,m)$ for some $y\in T$ and some integer $m<v(y)$. We claim that $m+1=v(y)$. Indeed, let $B$ be the quotient of $S(T_{n-1})$ by the subgroup generated by all elements $x'\in T_{n-1}$ which are not of the form $(y,k)$ for some $k<v(y)$. The image of $a'$ in $S(T_{n-1})/B$ must then be of the form 
		\[
		\sum_{k=0}^{v(y)-1}a'_{(y,k)}(y,k).
		\]
		But since $p\cdot a'=0$, the same must hold for its image in $S(T_{n-1})/B$. Therefore, since the quotient is isomorphic to a cyclic group, we must have that $a'_{(y,k)}=0$ for $k<v(y)-1$. Since $0\neq a'_{x'}=a'_{(y,m)}$, this implies that $m=v(y)-1$. 
		
		Now consider $$a=\sum_{x\in T}a'_xx \quad \text{ and } \quad a'-a=\sum_{x\in T_{n-1}\backslash T}a'_xx.$$ 
		Since $a\in A$ and $a'\in A'$ are in two different components of $S(T_n)\cong A\oplus A'$, we must have that $h(a'-a)=\operatorname{min}(h(a),h(a'))$ where $h$ denotes the $p$-height in $S(T_n)$.  Since $z\in T_{n-1}\backslash T$ was chosen to be of minimal $p$-height with $a'_z\neq 0$, we have that $h(a'-a)=h(z)$. We plan to show that $h(a)<h(z)$, thereby obtaining a contradiction. 
		
		First note that since $p\cdot a'=0$ and $a'_z\neq 0$, the set \[
		\{x'\in T_{n-1}\mid a'_{x'}\neq 0 \text{ and }p^k(x')=y \text{ for some }k>0\}
		\]
		must contain an element $\tilde{x}\neq z$. We claim that $h(\tilde{x})<h(z)$. Indeed, since $p^k(\tilde{x})=y$ for some $k>0$, we have $h(\tilde{x})+k\leq h(y)$. Since $h(z)=m=h(y)-1$ by the paragraph above,  it follows that 
		\[
		h(\tilde{x})\leq h(z)-k+1 \text{ for some }k>0
		\]
		Hence, if $h(\tilde{x})$ is not smaller than $h(z)$, then we must have that $k=1$. This implies that $p(\tilde{x})=y$ and $h(\tilde{x})=h(z)$. Therefore $\tilde{x}\in T$, since in the construction of $T_{n-1}$ from $T$ we add at most one element in $T_{n-1}\backslash T$ whose image under $p$ is $y$. Now since $v(y)=m+1$ we know that $z\notin T_{m}$ and $z\in T_{m+1}$ by construction of these sets. Also, since $v(\tilde{x})=m$, it must have height $m$ in $T_m$. Hence, $y=p(\tilde{x})$ will have height $m+1$ in $T_m$. But this  implies that we do not add any new element to $T_{m+1}$ whose image under $p$ is $y$, which is a contradiction to the existence of $z$. 
		
		It follows from the previous paragraph that $h(\tilde{x})<h(z)$. Since $z$ was an element of minimal height contained in $T_{n-1}\backslash T$ and satisfying $a'_z\neq 0$, we must have that $\tilde{x}\in T$. Since
		\[
		h(a)=\operatorname{min}\{h(x)\mid x\in T \text{ and }a'_x\neq 0\}
		\]
		we have that $h(a)\leq h(\tilde{x})<h(z)$. But $h(z)=h(a'-a)=\operatorname{min}(h(a),h(a'))$, so we have a contradiction. Hence $A'=(0)$, which proves the claim. 
	\end{proof}
	
	For $T\in \operatorname{ValTree}_n$, let $M_T$ be the morphism $S(T)\to S(T_{n-1})$ considered as an object in $\operatorname{sub}(\mathbb{Z}/(p^n))$. We can combine the previous results and results from \cite{HRW77} to obtain the following corollary.
	
	\begin{Corollary}\label{Corollary:IrretractableTreesSubmodules} 
		Let $T \in \operatorname{ValTree}_n$. The following hold:
		\begin{enumerate}
			\item\label{Corollary:IrretractableTreesSubmodules:1}  $M_T$ is indecomposable if and only if $T$ is irretractable.
			\item\label{Corollary:IrretractableTreesSubmodules:2} Assume $T$ is irretractable, and let $T'\in \operatorname{ValTree}_n$. If $M_T$ is isomorphic to $M_{T'}$, then $T$ is isomorphic to $T'$.
			\item\label{Corollary:IrretractableTreesSubmodules:3} 
			There exists a collection $T_1,\dots, T_m\in \operatorname{ValTree}_n$ of irretractable valuated trees, unique up to isomorphism and permutation, such that 
			\[
			M_T\cong M_{T_1}\oplus \dots M_{T_m}.
			\]
		\end{enumerate}
	\end{Corollary}
	
	\begin{proof}
		Note that  $\Phi(M_T)=S(T)$, and that by Theorem \ref{Theorem:ValuatedTreeToMonicRep} the objects $M_T$ has no nonzero summands in $\mathcal{Y}$. Hence, from the equivalence in Theorem \ref{Theorem:EquivalenceMonoValuatedGroupsFinGen} we get that $M_T$ is indecomposable if and only if $S(T)$ is indecomposable, that $M_T\cong M_{T'}$ if and only if $S(T)\cong S(T')$, and that $M_T\cong M_{T_1}\oplus \dots M_{T_m}$ if and only $S(T)\cong S(T_1)\oplus \dots \oplus S(T_m)$. The claim follows now from Theorem \ref{Theorem:IrretractableTreesValuatedGroups}
	\end{proof}

	Any $p$-valuated abelian group of the form $S(T)$ is called a \textit{simply presented }$p$\textit{-group}. Similarly, we call an object of the form $M_T \oplus X'$ with $T\in \operatorname{ValTree}_n$ and $X'\in \mathcal{Y}$ \textit{simply presented} monomorphic representation. These are precisely the monomorphic representations whose image under $\Phi$ is isomorphic to a simply presented $p$-group. In \cite{Ric88} they remark that any inclusion into a cyclic group and any split monomorphism is simply presented in $\operatorname{sub}(\mathbb{Z}/(p^n))$. Also, by \cite[Theorem 4.2]{HRW84} any object in  $\operatorname{sub}(\mathbb{Z}/(p^n))$ for $n\leq 4$ is simply presented. However, in \cite[Example 6]{HRW77} they construct an indecomposable representation in $\operatorname{sub}(\mathbb{Z}/(p^5))$ which is not simply presented, see also Example \ref{Example:p5TorsionSubgroups} below. For $n\geq 6$ the category $\operatorname{sub}(\mathbb{Z}/(p^n))$ is representation infinite by \cite{Pla76}. On the other hand, there are always finitely many indecomposable simply presented representations in $\operatorname{sub}(\mathbb{Z}/(p^n))$ by \cite[Lemma 5.4]{HW81}. Hence, there must be infinitely many indecomposable representations which are not simply presented when $n\geq 6$. We pose the following question:
	
	\begin{Question}
		Can one characterize the subcategory of $\operatorname{sub}(\mathbb{Z}/(p^n))$ of simply presented monomorphic representations intrinsically in $\operatorname{sub}(\mathbb{Z}/(p^n))$?
	\end{Question}

	There has been work on determining the lattice of irretractable valuated trees in $\operatorname{ValTree}_n$ for different $n$, see \cite{BHRW87,Ric88}. We pose the following question: 

	\begin{Question}
		What is the relationship between the lattice of irretractable valuated trees in $\operatorname{ValTree}_n$ and the Auslander--Reiten quiver of the subcategory of simply presented monomorphic representations over $\mathbb{Z}/(p^n)$?
	\end{Question}
	
	\begin{Remark}
		There has been a lot of work on $p$-valuated abelian groups which we do not mention here, both in the finite case, but also when the abelian group is infinite and there is no bound on the value $v(x)$ of the non-zero elements $x$ of the group. For the categorical theory of $p$-valuated groups see \cite{RW79}, and for results on simply presented $p$-valuated groups see \cite{HRW77}. For other references see \cite[Section 8.2]{Arn00} and \cite{HR81,BHW83,HRW84,BHRW87,Ric88}.
		
		%In \cite{HW81} they discuss the general structure of valuated $p$-groups. In particular, they investigate in which way different characterization of totally projective groups extend to valuated $p$-groups, they provide results on inclusions of valuated tree into $p$-groups, and they provide results on decompositions of simply presented valuated $p$-groups into sums of indecomposables. We refer to \cite{HRW77,HW81,HRW84} for some of the literature on this topic.
	\end{Remark}
	
	\subsection{Indecomposable $p^n$-torsion subgroups.}
	
	The classification of indecomposable objects in $\vartheta_p^3$ and $\vartheta_p^4$ was done in \cite{HRW84}, while the one for $\vartheta_p^5$ was done in \cite{RW99}. By Theorem \ref{Theorem:EquivalenceMonoValuatedGroupsFinGen}, this provides a classification of indecomposables in $\operatorname{sub}(\mathbb{Z}/(p^n))$ for $n\leq 5$. Here we give an explicit description of these indecomposables, using the equivalence in Theorem \ref{Theorem:EquivalenceMonoValuatedGroupsFinGen} and Corollary \ref{Corollary:IrretractableTreesSubmodules}.
	
	\begin{Example}[$p^3$-torsion subgroups]\label{Example:p3TorsionSubgroups} We consider $\vartheta_p^3$. By \cite[Corollary 4.3]{HRW84} the underlying abelian group of an indecomposable object in this category must by cyclic. Hence, all indecomposable objects are of the form $S(T)$ where $T$ is a cyclic tree. Therefore the irretractable valuated trees in $\operatorname{ValTree}_3$ are the ones of the form  
	\[
	\begin{tikzcd}
	i_1 
	\end{tikzcd} \quad \quad
	\begin{tikzcd}
	i_2 \arrow{d} \\
	j_2  
	\end{tikzcd} \quad \quad 
		\begin{tikzcd}
		i_3 \arrow{d} \\
		j_3 \arrow{d}  \\
		k_3 
	\end{tikzcd} 
	\]
	where $0\leq i_1\leq 2$ and $0\leq i_2<j_2\leq 2$ and $0\leq i_3<j_3<k_3\leq 2$. The corresponding sequences of parenthesized words are, respectively, 
	\[
	i_1 \quad \quad j_2i_2 \quad  \quad k_3j_3i_3.
	\]
	With this notation, $\operatorname{ValTree}_3$ has the following irretractable valuated trees
	\[
	0 \quad 1 \quad 2 \quad 10 \quad 20 \quad 21 \quad 210.
	\]
	By Theorem \ref{Theorem:ValuatedTreeToMonicRep}, these correspond to the indecomposable representations
 \begin{align*}
M_0=\mathbb{Z}/(p)\xrightarrow{\operatorname{id}}\mathbb{Z}/(p) && M_{10}=\mathbb{Z}/(p^2)\xrightarrow{\operatorname{id}}\mathbb{Z}/(p^2) && M_{210}=\mathbb{Z}/(p^3)\xrightarrow{\operatorname{id}}\mathbb{Z}/(p^3)
 \end{align*}
 given by the identity morphism, and
	\begin{align*}
	& M_1= \mathbb{Z}/(p)\xrightarrow{}\mathbb{Z}/(p^2) && M_2=\mathbb{Z}/(p)\xrightarrow{}\mathbb{Z}/(p^3) \\
 & M_{20}= \mathbb{Z}/(p^2)\xrightarrow{}\mathbb{Z}/(p)\oplus \mathbb{Z}/(p^3) 
	&& M_{21}=\mathbb{Z}/(p^2)\xrightarrow{}\mathbb{Z}/(p^3) 
	\end{align*}
 where $\mathbb{Z}/(p^i)\to \mathbb{Z}/(p^j)$ is given by the inclusion for $i<j$ and the projection for $i>j$. Together with the indecomposable representations 
	\begin{align*}
		& X_1=0\to \mathbb{Z}/(p) && X_2=0\to \mathbb{Z}/(p^2) && X_3=0\to \mathbb{Z}/(p^3)
	\end{align*}
in $\mathcal{Y}$, this gives all the indecomposable representations in $\operatorname{sub}(\mathbb{Z}/(p^3))$. There are $10$ of them in total.
	\end{Example}

\begin{Example}[$p^4$-torsion subgroups]\label{Example:p4TorsionSubgroups}By \cite[Theorem 4.2]{HRW84} the irretractable valuated trees in $\operatorname{ValTree}_4$ are the ones mentioned in Example \ref{Example:p3TorsionSubgroups}, together with the cyclic trees
\[
3\quad  30  \quad 31 \quad  32 \quad  310 \quad  320 \quad  321 \quad 3210	
\]
and the non-cyclic tree $3(2)(10)$. The tree $3210$ gives the representation 
\[
M_{3210}=\mathbb{Z}/(p^4)\xrightarrow{\operatorname{id}}\mathbb{Z}/(p^4)
\]
defined by the identity morphism. From the other cyclic trees we get the following indecomposable representations in  $\operatorname{sub}(\mathbb{Z}/(p^4))$ 
\begin{align*}
&M_3=\mathbb{Z}/(p)\xrightarrow{}\mathbb{Z}/(p^4) && M_{30}= \mathbb{Z}/(p^2)\xrightarrow{}\mathbb{Z}/(p)\oplus \mathbb{Z}/(p^4) \\
& M_{31}= \mathbb{Z}/(p^2)\xrightarrow{}\mathbb{Z}/(p^2)\oplus \mathbb{Z}/(p^4) && M_{32}=\mathbb{Z}/(p^2)\xrightarrow{}\mathbb{Z}/(p^4) \\
& M_{310}= \mathbb{Z}/(p^3)\xrightarrow{}\mathbb{Z}/(p^2)\oplus \mathbb{Z}/(p^4) && M_{320}= \mathbb{Z}/(p^3)\xrightarrow{}\mathbb{Z}/(p)\oplus \mathbb{Z}/(p^4) \\
& M_{321}=\mathbb{Z}/(p^3)\xrightarrow{}\mathbb{Z}/(p^4) 
\end{align*}
where any morphism $\mathbb{Z}/(p^i)\to \mathbb{Z}/(p^j)$ between any two components above is given by the inclusion when $i<j$, multiplication with $p$ when $i=j$, and the projection when $i>j$. 
From the non-cyclic tree we get the following representation
\[
M_{3(2)(10)}=\mathbb{Z}/(p)\oplus \mathbb{Z}/(p^3)\xrightarrow{\begin{pmatrix} \iota&\pi\\0&\iota \end{pmatrix}}\mathbb{Z}/(p^2)\oplus \mathbb{Z}/(p^4)
\]
where $\iota$ denotes the inclusion and $\pi$ the projection. Together with the object $$X_4=0\to \mathbb{Z}/(p^4)$$ in $\mathcal{Y}$ and the indecomposables in Example \ref{Example:p3TorsionSubgroups} this gives up to isomorphism all indecomposables in $\operatorname{sub}(\mathbb{Z}/(p^4))$. There are $20$ of them in total.
\end{Example}

\begin{Example}[$p^5$-torsion subgroups]\label{Example:p5TorsionSubgroups} In Section 4 \cite{RW99} they list the irretractable trees in $\operatorname{ValTree}_5$. In addition to the ones mentioned in Examples \ref{Example:p3TorsionSubgroups} and \ref{Example:p4TorsionSubgroups}, there are
\begin{align*}
& 4 \quad 40 \quad 41 \quad 42 \quad 43 \quad 410 \quad 420 \quad 421 \quad 430 \\ 
& 431 \quad 432 \quad 4210 \quad 4310 \quad 4320 \quad 4321 \quad 43210	\\
&43(2)(10) \quad  4(210)(32) \quad  4(210)(31) \quad 4(210)(30) \quad 4(210)(3) \\ 
& 4(30)(21) \quad 4(21)(3) \quad 4(20)(3) \quad 4(10)(3) \quad 4(10)(2) \\
& 4(32)(310)
\end{align*}
The tree $43210$ gives the representation 
\[
 M_{43210}=(\mathbb{Z}/(p^5)\xrightarrow{\operatorname{id}}\mathbb{Z}/(p^5))
\]
defined by the identity morphism. The cyclic trees $42$ and $431$ give the representations
\[
M_{42}=\mathbb{Z}/(p^2)\xrightarrow{\begin{pmatrix}
    p\cdot \iota\\ \iota
\end{pmatrix}}\mathbb{Z}/(p^3)\oplus\mathbb{Z}/(p^5) \quad \quad M_{431}=\mathbb{Z}/(p^3)\xrightarrow{\begin{pmatrix}
    p\cdot \pi\\ \iota
\end{pmatrix}}\mathbb{Z}/(p^2)\oplus \mathbb{Z}/(p^5)
\]
where $\pi$ is the projection and $\iota$ denote the canonical inclusions.
From the other cyclic trees we get the following indecomposable representations
\begin{align*}
&M_4=\mathbb{Z}/(p)\xrightarrow{}\mathbb{Z}/(p^5) && M_{40}= \mathbb{Z}/(p^2)\xrightarrow{}\mathbb{Z}/(p)\oplus \mathbb{Z}/(p^5) \\
& M_{41}= \mathbb{Z}/(p^2)\xrightarrow{}\mathbb{Z}/(p^2)\oplus \mathbb{Z}/(p^5) && M_{43}=\mathbb{Z}/(p^2)\xrightarrow{}\mathbb{Z}/(p^5) \\
& M_{410}= \mathbb{Z}/(p^3)\xrightarrow{}\mathbb{Z}/(p^2)\oplus \mathbb{Z}/(p^5) &&  M_{420}= \mathbb{Z}/(p^3)\xrightarrow{}\mathbb{Z}/(p)\oplus \mathbb{Z}/(p^3)\oplus \mathbb{Z}/(p^5)  \\
& M_{421}= \mathbb{Z}/(p^3)\xrightarrow{}\mathbb{Z}/(p^3)\oplus \mathbb{Z}/(p^5) && M_{430}= \mathbb{Z}/(p^3)\xrightarrow{}\mathbb{Z}/(p)\oplus \mathbb{Z}/(p^5) \\
 & M_{432}=\mathbb{Z}/(p^3)\xrightarrow{}\mathbb{Z}/(p^5) && M_{4210}=\mathbb{Z}/(p^4)\xrightarrow{}\mathbb{Z}/(p^3)\oplus\mathbb{Z}/(p^5) \\
 &  M_{4310}=\mathbb{Z}/(p^4)\xrightarrow{}\mathbb{Z}/(p^2)\oplus\mathbb{Z}/(p^5) && M_{4320}=\mathbb{Z}/(p^4)\xrightarrow{}\mathbb{Z}/(p)\oplus\mathbb{Z}/(p^5) \\
 & M_{4321}=\mathbb{Z}/(p^4)\xrightarrow{}\mathbb{Z}/(p^5)
\end{align*}
where any morphism $\mathbb{Z}/(p^i)\to \mathbb{Z}/(p^j)$ between any two components above is given by the inclusion when $i<j$, multiplication with $p$ when $i=j$, and the projection when $i>j$. We also have the representations coming from non-cyclic trees. A subset of them is given by an upper triangular $2\times 2$ matrix which have inclusions on the diagonal, and where the morphism $g\colon \mathbb{Z}/(p^i)\to \mathbb{Z}/(p^j)$ on the top right corner of the matrix is given by multiplication with $p$ if $i=j$, and by the projection map if $i>j$. They are
\begin{align*}
    & M_{43(2)(10)}=\mathbb{Z}/(p)\oplus \mathbb{Z}/(p^4)\xrightarrow{}\mathbb{Z}/(p^2)\oplus \mathbb{Z}/(p^5) \\
    & M_{4(210)(32)}=\mathbb{Z}/(p^2)\oplus \mathbb{Z}/(p^4)\xrightarrow{}\mathbb{Z}/(p^3)\oplus \mathbb{Z}/(p^5) \\
    & M_{4(210)(3)}=\mathbb{Z}/(p)\oplus \mathbb{Z}/(p^4)\xrightarrow{}\mathbb{Z}/(p^3)\oplus \mathbb{Z}/(p^5) \\
    & M_{4(21)(3)}=\mathbb{Z}/(p)\oplus \mathbb{Z}/(p^3)\xrightarrow{}\mathbb{Z}/(p^3)\oplus \mathbb{Z}/(p^5) \\ 
    & M_{4(10)(3)}=\mathbb{Z}/(p)\oplus \mathbb{Z}/(p^3)\xrightarrow{}\mathbb{Z}/(p^2)\oplus \mathbb{Z}/(p^5) 
\end{align*}
The remaining simply presented indecomposable representations in $\operatorname{sub}(\mathbb{Z}/(p^5))$ can be given as follows 
\begin{equation*}
     \resizebox{\textwidth}{!}{$M_{4(210)(31)}=\mathbb{Z}/(p^2)\oplus \mathbb{Z}/(p^4)\xrightarrow{\begin{pmatrix} 0&\pi\\\iota&\pi\\ 0& \iota \end{pmatrix}}\mathbb{Z}/(p^2)\oplus\mathbb{Z}/(p^3)\oplus \mathbb{Z}/(p^5) \quad \quad 
   M_{4(210)(30)}= \mathbb{Z}/(p^2)\oplus \mathbb{Z}/(p^4)\xrightarrow{\begin{pmatrix} \pi&0\\\iota&\pi\\ 0& \iota \end{pmatrix}}\mathbb{Z}/(p)\oplus\mathbb{Z}/(p^3)\oplus \mathbb{Z}/(p^5)$}
\end{equation*}
\begin{equation*}
    \resizebox{\textwidth}{!}{$M_{4(30)(21)}=\mathbb{Z}/(p^2)\oplus \mathbb{Z}/(p^3)\xrightarrow{\begin{pmatrix} \pi&\pi\\\iota&0\\ 0& \iota \end{pmatrix}}\mathbb{Z}/(p)\oplus\mathbb{Z}/(p^3)\oplus \mathbb{Z}/(p^5) \quad \quad 
M_{4(20)(3)}=\mathbb{Z}/(p)\oplus \mathbb{Z}/(p^3)\xrightarrow{\begin{pmatrix} 0&\pi\\\iota&p\cdot\\ 0& \iota \end{pmatrix}}\mathbb{Z}/(p)\oplus\mathbb{Z}/(p^3)\oplus \mathbb{Z}/(p^5)$}
\end{equation*}
\begin{equation*}
 \resizebox{\textwidth}{!}{$
     M_{4(10)(2)}=\mathbb{Z}/(p)\oplus \mathbb{Z}/(p^3)\xrightarrow{\begin{pmatrix} \iota&\pi\\0&p\cdot\\ 0& \iota \end{pmatrix}}\mathbb{Z}/(p^2)\oplus\mathbb{Z}/(p^3)\oplus \mathbb{Z}/(p^5) \quad \quad
    M_{4(32)(310)}=\mathbb{Z}/(p^2)\oplus \mathbb{Z}/(p^4)\xrightarrow{\begin{pmatrix} p\cdot&\pi\\\iota&0\\ 0 & \iota \end{pmatrix}}\mathbb{Z}/(p^2)\oplus\mathbb{Z}/(p^4)\oplus \mathbb{Z}/(p^5)$}
\end{equation*}
Here $\iota$ denotes the inclusion and $\pi$ the projection.
	
Finally, there are precisely two indecomposable $p$-valuated abelian groups in $\vartheta_p^5$ which are not simply presented. In \cite{RW99} they are represented by the hung forest $F$ and the hung tree $T$ below.
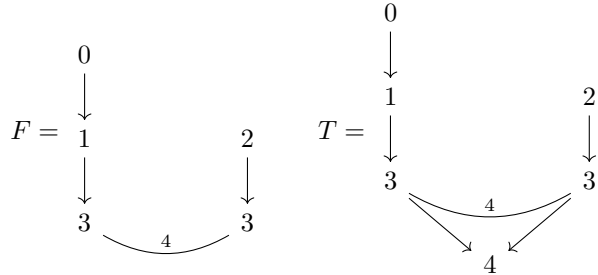
\[
 F= \begin{tikzcd}
 0 \arrow{d}\\
 1 \arrow{d} && 2 \arrow{d} \\
	3  \arrow[rr, bend right, dash,"4"]  && 3
	\end{tikzcd}\quad \quad 
T=\begin{tikzcd}
 0 \arrow{d}\\
 1 \arrow{d} && 2 \arrow{d} \\
	3 \arrow{rd} \arrow[rr, bend right, dash,"4"]  && 3 \arrow{ld} \\
	& 4  & 
	\end{tikzcd}
\]
The $p$-valuated abelian group $B_F$ associated to $F$ has generators $x$ and $y$ with relations $p^3x=0=p^2y$ and with $p$-valuation $v\colon B_F\to \mathbb{Z}_{\geq 0}$ uniquely defined by
\[
v(x)=0 \quad  v(px)=1 \quad  v(p^2x)=3 \quad v(y)=2 \quad  v(py)=3  \quad  v(p^2x-py)=4.
\]
For an explanation why it is not simply presented see \cite[Section 6]{HRW77}. 

Similarly, the $p$-valuated abelian group $B_T$ associated to $T$ has generators $x$ and $y$ with relations $p^3x=p^2y$ and $p^4x=0$, and with $p$-valuation $v\colon B_F\to \mathbb{Z}_{\geq 0}$ uniquely defined by
\begin{align*}
&v(x)=0  \quad v(px)=1  \quad   v(p^2x)=3 \quad  v(p^3x)=4\\
& v(y)=2 \quad v(py)=3   \quad v(p^2x-py)=4.
\end{align*}

A short computation shows that the unique indecomposable representation in $\operatorname{sub}(\mathbb{Z}/(p^5))$ which corresponds to $B_F$ is given by
\[
M_F=\mathbb{Z}/(p^2)\oplus \mathbb{Z}/(p^3)\xrightarrow{\begin{pmatrix} 0&\pi\\\iota&\iota\\ \iota& 0 \end{pmatrix}}\mathbb{Z}/(p^2)\oplus\mathbb{Z}/(p^4)\oplus \mathbb{Z}/(p^5)
\]
and the unique indecomposable representation in $\operatorname{sub}(\mathbb{Z}/(p^5))$ which corresponds to $B_T$ is given by
\[
M_T=\mathbb{Z}/(p^2)\oplus \mathbb{Z}/(p^4)\xrightarrow{\begin{pmatrix} p\cdot &\pi\\0&\iota\\ \iota& 0 \end{pmatrix}}\mathbb{Z}/(p^2)\oplus\mathbb{Z}/(p^5)\oplus \mathbb{Z}/(p^5).
\]
Together with the object $$X_5=0\to \mathbb{Z}/(p^5)$$ in $\mathcal{Y}$ and the indecomposables in Examples \ref{Example:p3TorsionSubgroups} and \ref{Example:p4TorsionSubgroups} we get up to isomorphism all indecomposables in $\operatorname{sub}(\mathbb{Z}/(p^5))$. There are $50$ of them in total.

	\end{Example}

	\section{Correction to the literature}\label{Section:Correction to the literature}

Here it is explained why $\operatorname{mono}(\mathtt{1}\to \mathtt{2}\to \mathtt{3}\to \mathtt{4},k([x]/(x^4))$ is wild, contradictory to \cite[Theorem 1.4]{Sim02}. In fact, wildness in the sense of Drozd is shown, i.e. that there exist a faithful exact functor
\[
    \operatorname{rep}\Gamma \to \operatorname{mono}(Q,\Lambda)
    \] 
that preserves indecomposables and reflects isomorphisms, where $\operatorname{rep}\Gamma$ is the category of finite-dimensional representations of $\Gamma$, the quiver with one vertex and two loops. The strategy is to first prove wildness for a subalgebra of a cover of $\operatorname{mono}(\mathtt{1}\to \mathtt{2}\to \mathtt{3}\to \mathtt{4},k([x]/(x^4))$, and then use this to show wildness of $\operatorname{mono}(Q,\Lambda)$. The statement and proof was kindly provided by the referee.

\begin{Theorem}\label{Theorem:RefereeWild}
    The category $\operatorname{mono}(\mathtt{1}\to \mathtt{2}\to \mathtt{3}\to \mathtt{4},k[x]/(x^4))$ has wild representation type.
\end{Theorem}

\begin{proof}
    Consider the category $\operatorname{Rep}B$ of representations of the $k$-linear category $B$ defined by the infinite quiver 
    \[
		\begin{tikzcd}
  \vdots \arrow[d] & \vdots \arrow[d] & \vdots \arrow[d] & \vdots \arrow[d] \\
		41 \arrow[d] \arrow{r} & 42 \arrow{r}\arrow{d} &43 \arrow{d}\arrow{r} & 44\arrow{d} \\
		31 \arrow[d] \arrow{r} & 32 \arrow{r}\arrow{d} &33 \arrow{d}\arrow{r} & 34\arrow{d} \\
  21 \arrow[d] \arrow{r} & 22 \arrow{r}\arrow{d} &23 \arrow{d}\arrow{r} & 24\arrow{d} \\
  11 \arrow[d] \arrow{r} & 12 \arrow{r}\arrow{d} &13 \arrow{d}\arrow{r} & 14\arrow{d}  \\
  \vdots & \vdots & \vdots & \vdots
		\end{tikzcd}
		\]
  with commutativity relations for all squares and nilpotency relation for composition of $4$ vertical arrows. Explicitly, the objects of $\operatorname{Rep}B$ are $k$-linear functors $B\to \operatorname{Mod}k$, i.e. are given by associating a $k$-vector space $M_{i,j}$ at each vertex $ij$ of $B$ and a linear transformation at each arrow of $B$ such that the commutativity and nilpotency relations hold. Note that $B$ is given by a Galois cover of $\Lambda Q$ in the sense of \cite[Section 3]{Gab81} where $\Lambda=k[x]/(x^4)$ and $Q=1\to 2\to 3\to 4$, and where $G=\mathbb{Z}$ acts on $B$ by translation vertically. Let $\operatorname{rep}B$ denote the representations $M$ of $B$ of total finite dimension, i.e. for which $M_{i,j}$ is finite-dimensional for all vertices $ij$ and non-zero for only finitely many vertices $ij$.  Consider the pushdown functor
  $$F_\lambda\colon \operatorname{rep}B\to \operatorname{rep}(Q,\operatorname{mod}\Lambda )$$
  Explicitly, for a vertex $j$ in $Q$ we have a $\Lambda=k[x]/(x^4)$-module $$(F_\lambda M)_j=\bigoplus_{i\in \mathbb{Z}}M_{i,j}$$ where the action of $x$ is given by the vertical maps $M_{i,j}\to M_{i-1,j}$. For an arrow $j\to j+1$ in $Q$ we have a map $(F_\lambda M)_j\to (F_\lambda M)_{j+1}$ of $\Lambda$-modules defined by the horizontal morphisms $M_{i,j}\to M_{i,j+1}$. Clearly, $F_\lambda$ is faithful and exact, and
  by \cite[Lemma 3.5]{Gab81} it also preserves indecomposables. Note that $F_\lambda M$ lies in $\operatorname{mono}(Q,\Lambda)$ if and only if the horizontal maps of $M$ are monic. %we have a faithful and exact functor $$F_\lambda\colon \operatorname{rep}B\to \operatorname{rep}(Q,\operatorname{mod}\Lambda )$$ preserving indecomposables. Explicitly, for $j\in Q_0$ we have a $\Lambda=k[x]/(x^4)$-module $$(F_\lambda M)_j=\bigoplus_{i\in \mathbb{Z}}M_{i,j}$$ where the action of $x$ is given by the vertical maps $M_{i,j}\to M_{i-1,j}$, and for an arrow $j\to j+1$ in $Q$ we have a map $(F_\lambda M)_j\to (F_\lambda M)_{j+1}$ of $\Lambda$-modules defined by the horizontal morphisms $M_{i,j}\to M_{i,j+1}$. In particular, $F_\lambda M$ lies in $\operatorname{mono}(Q,\Lambda)$ if and only if the horizontal maps of $M$ are monic.
  
  Now consider the subcategory of $\operatorname{rep}B$ given by representations of the form
\[
		\begin{tikzcd}
		 & & & 44\arrow{d} \\
		 &  &33 \arrow{d}\arrow{r} & 34\arrow{d} \\
  21 \arrow[d] \arrow{r} & 22 \arrow{r}\arrow{d} &23 \arrow{d}\arrow{r} & 24\arrow{d} \\
  11  \arrow[r,equal]& 12 \arrow[r,equal] &13 \arrow[r,equal] & 14 
		\end{tikzcd}
		\]
They can be identified with representations (or modules) of the path algebra $A_{44}$, given by the quiver depicted on the left 
\[
		\begin{tikzcd}
		 & & & 44\arrow{d} \\
		 &  &33 \arrow{d}\arrow{r} & 34\arrow{d} \\
  21 \arrow{r} & 22 \arrow{r} &23 \arrow{r} & 24\arrow{d} \\
  & && 14 
		\end{tikzcd} \quad \begin{tikzcd}
		 & & & 44\arrow{d} \\
		 &  &33 \arrow{d}\arrow{r} & 34\arrow{d} \\
  & 22 \arrow{r} &23 \arrow{r} & 24\arrow{d} \\
  & && 14 
		\end{tikzcd}
		\]
with commutativity relation for the square. Let $A_{43}$ be the subalgebra of $A_{44}$ with support in all vertices except $21$, depicted on the right. Consider the pullback functor 
\[
		\begin{tikzcd}
		 & & & M_{44}\arrow{d} \\
		 &  & & M_{34}\arrow{d} \\
  & M_{22} \arrow{r} &M_{23} \arrow{r} & M_{\tilde{33}} \arrow{d} \\
  &&& M_{24}\arrow[d] \\
  & &&  M_{14} 
		\end{tikzcd}\quad \mapsto \quad \begin{tikzcd}
		 & & & M_{44}\arrow{d} \\
		 &  &M_{33} \arrow{d}\arrow{r} & M_{34}\arrow{d} \\
  & M_{22} \arrow{r} &M_{23} \arrow{r} & M_{24}\arrow{d} \\
  & && M_{14} 
		\end{tikzcd}
  \]
  obtained by taking the pullback $M_{33}$ of $M_{34}\to M_{\tilde{33}}$ along $M_{23}\to M_{\tilde{33}}$, and then forgetting $M_{\tilde{33}}$. It is a functor $$\operatorname{rep}\tilde{\mathbb{E}}_6\to \operatorname{rep}A_{43}$$ where $\operatorname{rep}A_{43}$ and $\operatorname{rep}\tilde{\mathbb{E}}_6$ denotes the categories of finite-dimensional representations of $A_{43}$ and the $\tilde{\mathbb{E}}_6$-quiver with suitable orientation as depicted above. The functor can equivalently be described as the Hom-functor $$\operatorname{Hom}_{\tilde{\mathbb{E}}_6}(T,-)\colon \operatorname{rep}\tilde{\mathbb{E}}_6\to \operatorname{rep}A_{43} \quad \quad \text{where } T=N\oplus\bigoplus_{P\neq P_{\tilde{33}}}P. $$  Here $N$ is the unique indecomposable $\tilde{\mathbb{E}}_6$-module equal to $k$ at vertices $23,34,\tilde{33},24,14$, and $0$ otherwise, and the remaining sum is taken over all isomorphism classes of indecomposable projective $\tilde{\mathbb{E}}_6$-modules $P$ except the one with simple top concentrated at vertex $\tilde{33}$. Note that $T$ is a tilting module satisfying $\operatorname{End}_{\tilde{\mathbb{E}}_6}(T)^{\operatorname{op}}\cong A_{43}$, so $A_{43}$ is tilted of type $\tilde{\mathbb{E}}_6$. 

Let $\operatorname{mono}\tilde{\mathbb{E}}_6$ be the subcategory of $\operatorname{rep}\tilde{\mathbb{E}}_6$ consisting of representations where 
$$M_{22}\to M_{23}\quad \text{and} \quad M_{23}\to M_{\tilde{33}} \quad \text{and} \quad M_{23}\to M_{\tilde{33}}\to M_{24}$$
are monic. We claim that $\operatorname{mono}\tilde{\mathbb{E}}_6$ contains all the preprojective $\tilde{\mathbb{E}}_6$-modules. Indeed, a representation $M$ lies in $\operatorname{mono}\tilde{\mathbb{E}}_6$ if and only if $$\operatorname{Hom}_{\tilde{\mathbb{E}}_6}(S_{22},M)=\operatorname{Hom}_{\tilde{\mathbb{E}}_6}(S_{23},M)=\operatorname{Hom}_{\tilde{\mathbb{E}}_6}(T,M)=0$$ where $S_{22}$ and $S_{23}$ are the simples at vertex $22$ and $23$, respectively, and $T$ is the unique indecomposable $\tilde{\mathbb{E}}_6$-module equal to $k$ at vertices $23$ and $\tilde{33}$, and $0$ otherwise. Since $S_{22}$, $S_{23}$ and $T$ are all preinjective, the claim follows.

Next we claim that there exists indecomposable preprojective $\tilde{\mathbb{E}}_6$-modules $M$ with the dimension of $M_{22}$ arbitrary large. Indeed, first note that any $\tilde{\mathbb{E}}_6$-module is described by a morphism $P^n\to M'$ where $M'$ is the $\mathbb{E}_6$-module obtained from $M$ by deleting the vertex $22$, and $P$ is the projective $\mathbb{E}_6$-module with simple top at $23$, and $n$ is the dimension of $M_{22}$. Since non-isomorphic indecomposable preprojective modules have different dimension vectors (e.g. see \cite[Proposition 11.6 (iii)]{GLS17}), we can find $M$ indecomposable preprojective such that $\operatorname{dim}M=n+\operatorname{dim}M'$ is arbitrarily large, using that $\tilde{\mathbb{E}}_6$ is representation-infinite. Now define 
 $$s_n=\sum_K \operatorname{dim}(\operatorname{Hom}_{\mathbb{E}_6}(P^n,K))$$
 where the sum is over all isomorphism classes of indecomposable $\mathbb{E}_6$-modules $K$.  Note that the sum is finite since $\mathbb{E}_6$ is of finite representation type. Then, for any morphism $P^n\to M'$ where $M'$ has more than $s_n$ indecomposable summands, there must exists an indecomposable summand $M'_1$ of $M'$ such that $M'=M'_1\oplus M_2'$ and where $P^n\to M'_1$ factors through $P^n\to M'_2$, say via $g\colon M'_2\to M'_1$. Composing with the isomorphism 
\[
\begin{pmatrix}
1&-g\\
0&1
\end{pmatrix}\colon M'_1\oplus M'_2\xrightarrow{\cong} M'_1\oplus M'_2
\]  
we see that $M$ can be represented by a map $\begin{pmatrix}
f\\0
\end{pmatrix}\colon P^n\to M'_1\oplus M'_2$ for some $f$. But this implies that $M$ is not indecomposable. Hence, the dimension of $M_{22}$ must grow as the dimension of $M'$ grows. In particular, it can be arbitrarily large.

Now let $\mathcal{C}$ be the subcategory of $\operatorname{mono}\tilde{\mathbb{E}}_6$ where $M_{23}\oplus M_{34}\to M_{\tilde{33}}$ is epic, or equivalently where the Ext-group $$\operatorname{Ext}^1_{\tilde{\mathbb{E}}_6}(T,M)\cong \operatorname{Ext}^1_{\tilde{\mathbb{E}}_6}(N,M)$$
is $0$. Note that there are only finitely many indecomposable representations which are not in $\mathcal{C}$, since $N$ is a preprojective module. Hence, by the previous paragraph we can find an indecomposable preprojective $\tilde{E}_6$-modules $M$ in $\mathcal{C}$ for which $d=\operatorname{dim}M_{22}\geq 3$. Note that the pullback functor gives an equivalence $\mathcal{C}\xrightarrow{\cong}\operatorname{mono}A_{43}$ where $\operatorname{mono}A_{43}$ is the subcategory of $\operatorname{rep}A_{43}$ where 
$$M_{22}\to M_{23}\quad \text{and} \quad M_{33}\to M_{34}\quad \text{and} \quad M_{23}\to M_{24}$$ are monic. Let $X\in \operatorname{mono}A_{43}$ be the image $M$ under this equivalence. Since $M$ is preprojective we have that $\operatorname{End}_{\tilde{\mathbb{E}}_6}(M)=k$, and so $\operatorname{End}_{A_{43}}(X)=k$. Consider $X$ as an $A_{44}$-module, and let $Y$ be the simple $A_{44}$-module concentrated at vertex $21$. Note that $X,Y$ form an orthogonal pair in the sense that $$\operatorname{Hom}_{A_{44}}(X,Y)=0=\operatorname{Hom}_{A_{44}}(Y,X)\quad \text{and} \quad \operatorname{End}_{A_{44}}(X)=k=\operatorname{End}_{A_{44}}(Y).$$ Since $\operatorname{Ext}^1_{A_{44}}(X,Y)=0$ and $\operatorname{Ext}^1_{A_{44}}(Y,X)=d$, we have a fully faithful exact functor $$F\colon \operatorname{rep}K_d\to \operatorname{rep}A_{44}$$ by \cite[1.5 Lemma]{Rin76}, where $K_{d}$ is the path algebra of the $d$th Kronecker quiver, i.e. the quiver with two vertices $1$ and $2$ and $d$ arrows from $1$ to $2$.  The representations in the image of $F$ have the following property. If the $K_d$-module $L$ has dimension type $(a,b)$, there is a short exact sequence 
\[
0\to X^b\to F(L)\to Y^a\to 0
\]
of $A_{44}$-modules. Note that the horisontal maps for the representation $F(L)$ are monic, unless $F(L)$ contains the simple injective module $Y$ as a summand. 

Composing $F$ with the inclusion of $\operatorname{rep}A_{44}$ into $\operatorname{rep}B$ where $B$ is the path algebra of the infinite quiver above, we get a fully faithful exact functor $\tilde{F}\colon \operatorname{rep}K_d\to \operatorname{rep}B$. Note that non-isomorphic $K_{d}$-modules will be sent to different orbits under the action of $G=\mathbb{Z}$ by vertical translation. Hence, by \cite[Lemma 3.5]{Gab81} the composite
\[
\operatorname{rep}K_d\xrightarrow{\tilde{F}} \operatorname{rep}B\xrightarrow{F_{\lambda}} \operatorname{rep}(Q,\operatorname{mod}\Lambda)
\] 
reflects isomorphisms. We also know $F_\lambda$  is faithful, exact, and preserves indecomposables, so the same must hold for $F_\lambda\circ \tilde{F}$. Note that the image of a $K_d$-module $L$ lies in $\operatorname{mono}(Q,\Lambda)$ if and only if $L$ does not contain the simple injective $K_d$-module as a summand. Since $K_d$ is a wild algebra for $d\geq 3$, we can find a fully faithful exact functor $\operatorname{rep}\Gamma\to \operatorname{rep}K_d$ where $\Gamma$ is the quiver with one vertex and two loops. For a construction see e.g. \cite[Lemma 10.2.2]{Kra08}. Since $\operatorname{rep}\Gamma$ has no simple injective object, the image of this functor cannot contain the simple injective $K_d$-module. Hence, the composite 
\[
\operatorname{rep}\Gamma\to \operatorname{rep}K_d\xrightarrow{\tilde{F}}\operatorname{rep}B\xrightarrow{F_\lambda} \operatorname{rep}(Q,\operatorname{mod}\Lambda)
\]
must land in $\operatorname{mono}(Q,\Lambda)$. Since it is faithful, exact, preserves indecomposables, and reflects isomorphisms, this shows that $\operatorname{mono}(Q,\Lambda)$ is wild.
\end{proof}

	%------
	% Insert acknowledgments and information
	% regarding funding at the end of the last
	% section, i.e., right before the bibliography.
	%------
	
\section*{Acknowledgements}

I would like to thank the Centre for Advanced Study in Oslo, where a large part of this manuscript was written up. I am also grateful to Julian Külshammer for computing Example \ref{Example:D_n4OverCyclicLength2}, and to Julian Külshammer and Chrysostomos Psaroudakis for comments on an earlier version of the manuscript. I would like to thank Markus Schmidmeier for letting me know about the references \cite{Rum86,Rum90}. Finally, I am grateful to the anonymous referee for several helpful comments and corrections, and for providing the statement and proof of Theorem \ref{Theorem:RefereeWild}.

\bibliographystyle{alpha}
\bibliography{publication}

\end{document}